\documentclass[final,leqno,letterpaper]{etna}

\setbibdata{1}{xx}{46}{2017} 

\hypersetup{%
    pdftitle={Preconditioning without a preconditioner using randomized block Krylov subspace methods},
    pdfauthor={Tyler Chen, Caroline Huber, Ethan Lin, Hajar Zaid},
    pdfkeywords={preconditioning, randomized, conjugate gradient, block Krylov subspace methods}
    }

\title{Preconditioning without a preconditioner using randomized block Krylov subspace methods\thanks{%
Received... Accepted... Published online on... Recommended by....
}}

\author{Tyler Chen\footnotemark[2]
        \and Caroline Huber\footnotemark[3] \and Ethan Lin\footnotemark[3] \and Hajar Zaid\footnotemark[4]}

\shorttitle{Preconditioning without a preconditioner} 
\shortauthor{T. CHEN, C. HUBER, E. LIN, H. ZAID}

\usepackage{xcolor}
\definecolor{c0}{HTML}{1e3264}
\definecolor{c1}{HTML}{c82336}

\usepackage{amsmath,amssymb}

\usepackage{hyperref}
\hypersetup{
    colorlinks,
    linkcolor=c0,
    citecolor=c1,
    urlcolor=c0
}

\usepackage{float}
\usepackage{subcaption}
\usepackage[nameinlink,capitalize,noabbrev]{cleveref}

\crefformat{equation}{#2(#1)#3}
\crefmultiformat{equation}{#2(#1)#3}%
{ and~#2(#1)#3}{, #2(#1)#3}{ and~#2(#1)#3}
\crefrangeformat{equation}{(#3#1#4) to (#5#2#6)}

\crefname{section}{Section}{Sections}
\crefname{subsection}{Subsection}{Subsections}

\crefname{corollary}{Corollary}{Corollaries}
\crefname{proposition}{Proposition}{Propositions}
\crefname{lemma}{Lemma}{Lemmas}

\usepackage{enumitem}

\usepackage{bm}

\renewcommand{\vec}{\mathbf}

\renewcommand{\d}{\mathrm{d}}

\newcommand{\cg}[1]{\mathsf{cg}_{#1}}
\newcommand{\pcg}[1]{\mathsf{pcg}_{#1}}
\newcommand{\bcg}[2]{\mathsf{bcg}_{#1}^{(#2)}}

\newcommand{\lansq}[1]{\mathsf{sq}_{#1}}
\newcommand{\blansq}[2]{\mathsf{bsq}_{#1}^{(#2)}}
\newcommand{\bLansq}[1]{\mathsf{bsq}_{#1}}

\newcommand{\blanisq}[2]{\mathsf{bisq}_{#1}^{(#2)}}
\newcommand{\bLanisq}[1]{\mathsf{bisq}_{#1}}

\newcommand{\CG}{\hyperref[def:CG]{CG}}
\newcommand{\BCG}{\hyperref[def:BCG]{block-CG}}
\newcommand{\PCG}{\hyperref[def:PCG]{preconditioned-CG}}

\usepackage{makecell}

\newcommand{\EE}{\operatorname{\mathbb{E}}}
\newcommand{\PP}{\operatorname{\mathbb{P}}}

\newcommand{\T}{\mathsf{T}}

\newcommand{\tr}{\operatorname{tr}}

\newcommand{\llbracket}{[\![}
\newcommand{\rrbracket}{]\!]}

\usepackage{nicematrix}

\begin{document}

\maketitle

\renewcommand{\thefootnote}{\fnsymbol{footnote}}

\footnotetext[2]{JPMorganChase (work on this paper was initiated while at New York University)}
\footnotetext[3]{New York University}
\footnotetext[4]{Graduate Center, CUNY}

\begin{abstract}
We describe a randomized variant of the block conjugate gradient method for solving a single positive-definite linear system of equations.
This method provably outperforms preconditioned conjugate gradient with a broad-class of Nystr\"om-based preconditioners, without ever explicitly constructing a preconditioner.
In analyzing our algorithm, we derive theoretical guarantees for new variants of Nystr\"om preconditioned conjugate gradient which may be of separate interest.
We also describe how our approach yields fast algorithms for key data-science tasks such as computing the entire ridge regression regularization path and generating multiple independent samples from a high-dimensional Gaussian distribution.
\end{abstract}



\section{Introduction}
\label{sec:intro}

Solving the regularized linear system
\begin{equation}
\label{eqn:regularized_ls}
\vec{A}_\mu \vec{x} = \vec{b}, \quad \vec{A}_\mu := \vec{A}+\mu\vec{I},
\end{equation}
where $\vec{A}\in\mathbb{R}^{d\times d}$ is symmetric positive definite and $\mu\geq 0$ is a critical task across the computational sciences.
Systems of the form \cref{eqn:regularized_ls} arise in a variety of settings, including the following:

\textit{Positive definite linear systems.}
When $\mu=0$, \cref{eqn:regularized_ls} is simply the task of solving a positive definite linear system $\vec{A}\vec{x} = \vec{b}$, one of the most common problems in numerical linear algebra.

\textit{Ridge regression.}
Given a data matrix $\vec{Z}\in\mathbb{R}^{n\times d}$, a source term $\vec{f}\in\mathbb{R}^n$, and a regularization parameter $\mu\geq 0$, the ridge regression problem\footnote{Ridge regression is a special case of Tikhonov regularization, where $\mu \|\vec{x}\|^2$ is replaced by a somewhat more general regularization term $\|\vec{\Gamma}\vec{x}\|^2$ \cite{engl_hanke_neubauer_96}.} is
\begin{equation}\label{eqn:ridge_regression}
    \min_{\vec{x}\in\mathbb{R}^{d}} \left( \| \vec{Z}\vec{x} - \vec{f} \|^2  + \mu \|\vec{x}\|^2 \right).
\end{equation}
A direct computation shows that the solution to \cref{eqn:ridge_regression} is also the solution to
\begin{equation}
    (\vec{Z}^\T\vec{Z} + \mu \vec{I})\vec{x} = \vec{Z}^\T\vec{f};
\end{equation}
i.e. a system of the form \cref{eqn:regularized_ls} with $\vec{A} = \vec{Z}^\T\vec{Z}$ and $\vec{b} = \vec{Z}^\T\vec{f}$.
In a number of applications we are interested in the {whole regularization path}; i.e. $\vec{x} = \vec{x}_{\mu}$ for all values $\mu\geq 0$; e.g. in order to perform cross-validation to select a value of $\mu$ to use for future predictions.

\textit{Sampling Gaussians.}
If $\vec{b}$ has independent standard normal entries, then $\bm{\mu} + \vec{A}^{1/2}\vec{b}$ is a Gaussian vector with mean $\bm{\mu}$ and covariance $\vec{A}$. 
In order to approximate $\vec{A}^{1/2}\vec{b}$, it is common to make use of the identity
\begin{equation}\label{eqn:sqrt_rational}
    \vec{A}^{1/2}\vec{b}
    = \frac{2}{\pi}\int_{0}^\infty \vec{A}(\vec{A}+z^2\vec{I})^{-1}\vec{b}\, \d z.
\end{equation}
The term $(\vec{A}+z^2\vec{I})^{-1}\vec{b}$ in the integrand of \cref{eqn:sqrt_rational} is of the form \cref{eqn:regularized_ls} with $\mu = z^2$.
Often we are interested in sampling several Gaussian vectors with the same covariance matrix.

\begin{figure}[htb]
    \centering
    \includegraphics[scale=.6,trim={.75cm 0 1cm 0},clip]{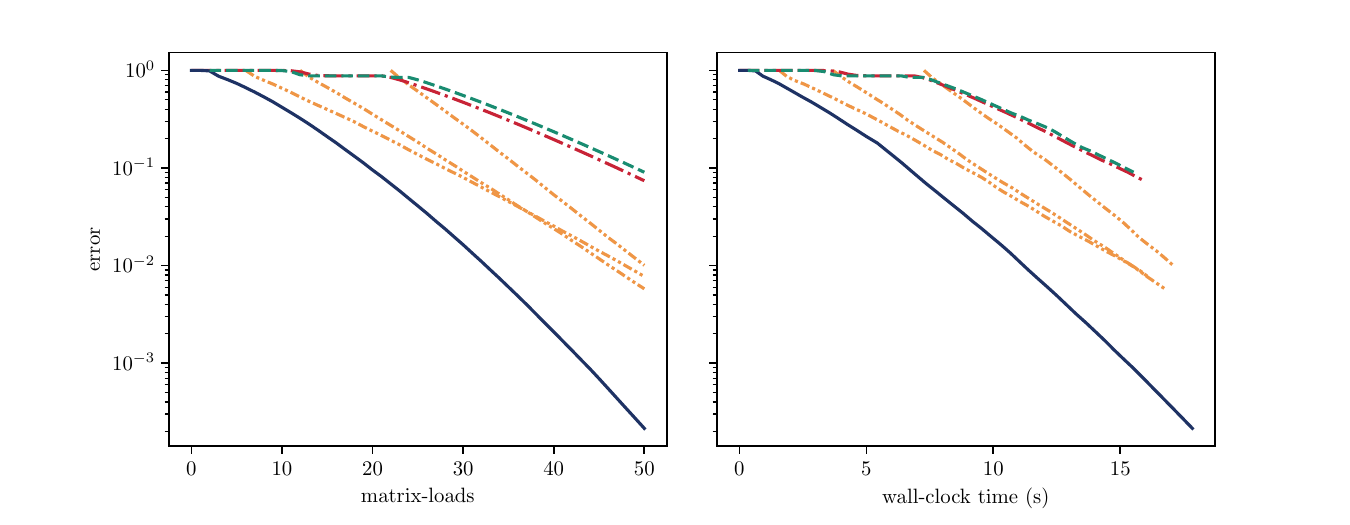}
    \caption{
    Relative error $\|\vec{A}^{-1}\vec{b} - \mathsf{alg}\|_{\vec{A}} / \|\vec{A}^{-1}\vec{b}\|_{\vec{A}}$ in terms of matrix-loads (left) and wall-clock time (right) for our proposed randomized variant of the block conjugate gradient method ({\protect\raisebox{0mm}{\protect\includegraphics[scale=0.68]{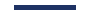}}}),
    standard conjugate gradient
    ({\protect\raisebox{0mm}{\protect\includegraphics[scale=0.68]{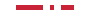}}}), Nystr\"om preconditioned conjugate gradient from \cite{frangella_tropp_udell_23} ({\protect\raisebox{0mm}{\protect\includegraphics[scale=0.68]{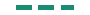}}}), and generalizations Nystr\"om preconditioned conjugate gradient using higher-depth Nystr\"om approximations ({\protect\raisebox{0mm}{\protect\includegraphics[scale=0.68]{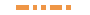}}}).
    Our method outperforms all these methods without the need for selecting hyperparameters (see \cref{thm:main}), which may be difficult to do effectively in practice.
    In particular, we store $\vec{A}$ in 8 separate $1000\times 8000$ chunks and perform (block) matrix-vector products with $\vec{A}$ by sequentially loading a single chunk from the disk into random access memory and performing the appropriate part of the products.
    The runtime is dominated by the cost of loading chunks of the matrix into memory, so the wall-clock-time is nearly proportional to matrix-loads.
    Full experiment description in \cref{sec:numerical:convergence}.
    }
    \label{fig:intro_fig}
\end{figure}

\subsection{(Block) Krylov subspace methods}

All three problems described above are commonly solved using a class of algorithms called Krylov Subspace Methods (KSMs); see e.g., \cite{saad_03,hanke_17,chen_24}.
KSMs make use of the so-called Krylov subspace
\begin{equation*}
\mathcal{K}_t(\vec{A},\vec{b})
:= \operatorname{span}\{\vec{b}, \vec{A}\vec{b}, \ldots, \vec{A}^{t-1}\vec{b}\},
\end{equation*}
which can be efficiently computed using matrix-vector products with $\vec{A}$.
Since it is often of interest to solve \cref{eqn:regularized_ls} for multiple values of $\mu$, we repeatedly make use of the fact that the Krylov subspace is shift-invariant; i.e. that 
\begin{equation}\label{eqn:shift_invariance}
    \forall \mu\in\mathbb{C}: \mathcal{K}_t(\vec{A}_\mu,\vec{b}) = \mathcal{K}_t(\vec{A},\vec{b}).
\end{equation}

In this paper, we advocate for the use of generalization of KSMs called block-KSMs.
Given a matrix $\vec{B}\in\mathbb{R}^{d\times m}$ (typically $m\ll d$) with columns $\vec{b}^{(1)}, \ldots, \vec{b}^{(m)}$, the block Krylov subspace is defined as
\begin{equation}\label{eqn:block_krylov_def}
    \mathcal{K}_t(\vec{A},\vec{B}) := 
    \mathcal{K}_t(\vec{A},\vec{b}^{(1)}) + \cdots + \mathcal{K}_t(\vec{A},\vec{b}^{(m)}).
\end{equation}
That is, $\mathcal{K}_t(\vec{A},\vec{B})$ is the space consisting of all linear combinations of vectors in the Krylov subspaces $\mathcal{K}_t(\vec{A},\vec{b}^{(1)}), \ldots, \mathcal{K}_t(\vec{A},\vec{b}^{(m)})$.
Analogous to \cref{eqn:shift_invariance}, the block Krylov subspace is shift invariant; i.e. $\mathcal{K}_t(\vec{A}_\mu,\vec{B}) = \mathcal{K}_t(\vec{A},\vec{B})$ for any shift $\mu$.

As we discuss in \cref{sec:comp_assm}, compared to their single-vector counterparts, block-KSMs have a number of computational benefits on modern computational architectures. 
As such, block-KSMs are widely used for tasks in numerical linear algebra such as low-rank approximation and eigenvalue approximation, where they are known to satisfy strong theoretical guarantees \cite{golub_underwod_77,saad_80,halko_martinsson_tropp_11,musco_musco_15,li_linderman_slam_stanton_kluger_tygert_17,martinsson_tropp_20,tropp_webber_23}.

Interestingly, the use of block-KSMs for solving problems like \cref{eqn:regularized_ls} is more limited (see \cref{sec:past_work} for a discussion), and single-vector methods such as the conjugate gradient algorithm are more popular. 
A high-level aim of this paper is to encourage further study on the use of block-KSMs for tasks such as \cref{eqn:regularized_ls}. 
In particular, we believe that techniques from randomized numerical linear algebra have the potential to provide new insights and theoretical guarantees for block-KSMs used to solve \cref{eqn:regularized_ls}.

\subsection{Computational assumptions}
\label{sec:comp_assm}
Throughout this paper, we assume $\vec{A}$ is accessed only through matrix-vector products $\vec{x} \mapsto \vec{A}\vec{x}$ or block matrix-vector products $[\vec{x}_1~\cdots~\vec{x}_m]\mapsto [\vec{A}\vec{x}_1~\cdots~\vec{A}\vec{x}_m]$.
While one can simulate a block matrix-vector product using $m$ matrix-vector products, in many settings the cost of block products is often nearly independent of the block size $m$ (so long as $m$ is not too large). 
For example, this can be the case when $\vec{A}$ is so large that it must be loaded in chunks from slow memory (as in \cref{fig:intro_fig}), 
where $\vec{A}$ is the Hessian of a large neural-network and matrix-vector products are performed using ``Pearlmutter's trick'' \cite{pearlmutter_94}, or where $\vec{A}$ corresponds to the solution operator of an integral equation and is applied via a solver fast direct solver \cite{martinsson_oneil_25}.
Allowing for more efficient data-access is widely understood as one of the major benefits of many randomized linear algebra algorithms \cite{martinsson_tropp_20,tropp_webber_23}.

The algorithms we propose in this paper are designed to economize the number of times $\vec{A}$ is loaded into memory (henceforth referred to as matrix-loads), and not other costs such as matrix-vector products, floating point operations, or storage. 
In settings where matrix-loads are not the dominant cost, the proposed algorithms may not yield significant advantages.
The costs of the algorithm from this paper are discussed more in \cref{sec:runtime}, and we explore them in the numerical experiments in \cref{sec:numerical}.
A detailed understanding of tradeoffs in various costs is somewhat beyond the scope of this paper, but is of practical importance. 

\subsection{Notation}

We denote the eigenvalues of $\vec{A}$ by $\lambda_1\geq \cdots\geq \lambda_d$.
We write $\kappa = \lambda_1/\lambda_d$ and $\kappa_{r+1}(\mu) := (\lambda_{r+1}+\mu)/(\lambda_d+\mu)$, as the condition number of $\vec{A}_\mu$ with the top-$r$ eigenvalues removed.
The eigenvalues of an arbitrary positive definite matrix $\vec{M}$ are $\lambda_1(\vec{M})\geq \cdots \geq \lambda_d(\vec{M})$. 
The condition number of $\vec{M}$ is $\kappa(\vec{M}) := \lambda_1(\vec{M}) / \lambda_d(\vec{M})$ and the spectral norm is $\|\vec{M}\| := \lambda_1(\vec{B})$.
The $\vec{M}$-norm of a vector $\vec{x}$ is $\|\vec{x}\|_{\vec{M}} := \| \vec{M}^{1/2} \vec{x} \| = \sqrt{\vec{x}^\T \vec{M} \vec{x}}$.

\subsection{Organization}

In \cref{sec:bg} we describe key background on Krylov subspace methods and preconditioning.
Our algorithm and main conceptual results are described in \cref{sec:ABCG}.
Explicit and simple to use probabilistic results are proved in \cref{sec:analysis}.
We discuss our bounds in the context of past work in \cref{sec:past_work} and show how the theory we develop can be used to sample Gaussian vectors in \cref{sec:gaussian_sample}.
Finally, numerical experiments are provided in \cref{sec:numerical}.

\section{Background}
\label{sec:bg}
In this section we provide an overview of a number of relevant algorithms.
The methods we discuss in this section are standard, and for mathematical simplicity, we define them by their optimality conditions rather than algorithmically (all of our bounds assume exact arithmetic); see e.g., \cite{greenbaum_97,saad_03,meurant_06} for comprehensive treatments of these methods.

\subsection{Preconditioned Conjugate Gradient}

The preconditioned conjugate gradient algorithm (PCG) is a powerful KSM for solving positive definite linear systems. 
PCG makes use of a symmetric positive definite preconditioner $\vec{P}_{\mu}$ that transforms the system $\vec{A}_{\mu}\vec{x} = \vec{b}$ into the system 
\begin{equation*}
    (\vec{P}_{\mu}^{-1/2} \vec{A}_{\mu} \vec{P}_{\mu}^{-1/2}) \vec{y} = \vec{P}_{\mu}^{-1/2}\vec{b}
    ,\quad
    \vec{P}_{\mu}^{1/2}\vec{x} = \vec{y},
\end{equation*}
and works over a transformed Krylov subspace
\begin{equation}
\label{eqn:precond_krylov_def}
\mathcal{K}_t(\vec{A}_{\mu},\vec{b};\vec{P}_{\mu})
:= 
\mathcal{K}_t(\vec{P}_{\mu}^{-1}\vec{A}_{\mu},\vec{P}_{\mu}^{-1}\vec{b}).
\end{equation}
Specifically, PCG is defined defined by an optimality condition.
\begin{definition}\label{def:PCG}
    The $t$-th PCG iterate corresponding to a positive definite preconditioner $\vec{P}_{\mu}$ is defined as 
    \begin{equation*}
    \pcg{t}(\mu) := \operatornamewithlimits{argmin}_{\vec{x}\in\mathcal{K}_t(\vec{A}_{\mu},\vec{b};\vec{P}_{\mu})} \| \vec{A}_{\mu}^{-1}\vec{b} - \vec{x} \|_{\vec{A}_{\mu}}.
    \end{equation*}
\end{definition}
The iterate $\pcg{t}(\mu)$ can be computed using $t-1$ matrix-vector products with $\vec{A}$ (and $t$ products with $\vec{P}_{\mu}^{-1}$). 

The simplest choice of preconditioner is $\vec{P}_\mu = \vec{I}$, which yields the standard conjugate gradient algorithm (CG) \cite{hestenes_stiefel_52}. 
\begin{definition}\label{def:CG}
    The $t$-th CG iterate is defined as 
    \begin{equation*}
    \cg{t}(\mu) := \operatornamewithlimits{argmin}_{\vec{x}\in\mathcal{K}_t(\vec{A}_{\mu},\vec{b})} \| \vec{A}_{\mu}^{-1}\vec{b} - \vec{x} \|_{\vec{A}_{\mu}}.
    \end{equation*}
\end{definition}

\subsubsection{Convergence}
PCG (and hence CG) satisfy a well-known convergence guarantee in terms of the condition number of $\vec{P}_{\mu}^{-1/2}\vec{A}_{\mu}\vec{P}_{\mu}^{-1/2}$; see e.g. \cite{greenbaum_97}.\footnote{Here we use that $(({\sqrt{\kappa}-1})/({\sqrt{\kappa}+1}))^t \leq \exp(-2t/\sqrt{\kappa})$ for all $\kappa\geq1$ in order to obtain a simpler expression than the typical bound. This approximation is pessimistic when $\kappa\approx 1$, but it's clear from our analysis that other bounds for CG can be used instead.}

\begin{corollary}\label[corollary]{thm:pcg_condno_bd}
Let $\vec{P}_\mu$ be any preconditioner. 
Then the $t$-th \PCG{} iterate corresponding to the preconditioner $\vec{P}_\mu$ satisfies
\begin{equation*}
    \frac{\| \vec{A}_\mu^{-1}\vec{b} - \pcg{t}(\mu) \|_{\vec{A}_\mu}}{\| \vec{A}_\mu^{-1}\vec{b}\|_{\vec{A}_\mu}}
    \leq 2 \exp\left(-\frac{2t}{\sqrt{\kappa(\vec{P}_{\mu}^{-1/2}\vec{A}_{\mu}\vec{P}_{\mu}^{-1/2})}}\right).
\end{equation*}
\end{corollary}
\Cref{thm:pcg_condno_bd} implies that if we can find $\vec{P}_\mu$ so that $\kappa(\vec{P}_{\mu}^{-1/2}\vec{A}_{\mu}\vec{P}_{\mu}^{-1/2})$ is small, then PCG converges rapidly. 
The choice of $\vec{P}_\mu$ minimizing this condition number is of course $\vec{P}_\mu = \vec{A}_\mu$, but this is not a practical preconditioner; if we knew $\vec{A}_\mu^{-1}$ we could easily compute the solution $\vec{A}_\mu^{-1}\vec{b}$.
Thus, finding a suitable choice of $\vec{P}_\mu$, which balances improvements to the convergence of PCG with the cost of building/applying $\vec{P}_\mu$ is critical.

\subsection{Deflation preconditioner}

If $\vec{A}$ is poorly conditioned due to the presence of $r$ eigenvalues much larger than the remaining $n-r$ eigenvalues, then we might hope to learn a good approximation of the top $r$ eigenvalues and ``correct'' this ill-conditioning. 
Towards this end, suppose we have a good rank-$r$ approximation $\vec{U}\vec{D}\vec{U}^\T$ to $\vec{A}$, where $\vec{U}$ has orthonormal columns and $\vec{D}$ is diagonal.
Intuitively, this low-rank approximation contains the information needed to remove the top eigenvalues of $\vec{A}$, thereby reducing the condition number of the preconditioned system.
In particular, one can form the preconditioner 
\begin{equation}\label{eqn:precond_def}
\vec{P}_{\mu}
:= \frac{1}{\theta + \mu} \vec{U}(\vec{D}+\mu\vec{I})\vec{U}^\T + (\vec{I} - \vec{U}\vec{U}^\T),
\end{equation}
where $\theta > 0$ is a parameter that must be chosen along with the factorization $\vec{U}\vec{D}\vec{U}^\T$.

It is not hard to verify that 
\begin{equation*}
\vec{P}_{\mu}^{-1}
= (\theta + \mu) \vec{U}(\vec{D}+\mu\vec{I})^{-1} \vec{U}^\T + (\vec{I} - \vec{U}\vec{U}^\T).
\end{equation*}
This means applying $\vec{P}_{\mu}^{-1}$ to a vector can be done in  $O(dr)$ arithmetic operations, which is relatively cheap if $r\ll d$.
Hence, so long as a reasonable factorization $\vec{U}\vec{D}\vec{U}^\T$ can be obtained, \cref{eqn:precond_def} can be used as a preconditioner; see \cite{frank_vuik_01,gutknecht_12} and the references within.

\subsubsection{Exact deflation}

Traditionally, it has been suggested to take $\vec{U}\vec{D}\vec{U}^\T$ as $\llbracket \vec{A}\rrbracket_r$, the \emph{best} rank-$r$ approximation to $\vec{A}$; i.e. to set $\vec{D}$ as a diagonal matrix containing the top $r$ eigenvalues of $\vec{A}$ and $\vec{U}$ the corresponding eigenvectors.
In this case \cref{eqn:precond_def} is sometimes called a spectral deflation preconditioner.
The name arises because the eigenvalues of $\vec{P}_{\mu}^{-1/2} \vec{A}_\mu \vec{P}_{\mu}^{-1/2}$ are 
\begin{equation*}
\underbrace{\theta+\mu, \ldots, \theta+\mu}_{\text{$r$-copies}}, \lambda_{r+1}+\mu, \ldots, \lambda_d+\mu,
\end{equation*}
and hence, if $\theta\in[\lambda_d,\lambda_{r+1}]$, then the condition number of $\vec{P}_\mu^{-1/2}\vec{A}_\mu\vec{P}_\mu^{-1/2}$ is 
$\kappa_{r+1}(\mu) = (\lambda_{r+1}+\mu)/(\lambda_d+\mu)$.
In other words, the top $r$ eigenvalues of $\vec{A}_\mu$ are deflated. 
In combination with \cref{thm:pcg_condno_bd}, this yields the following convergence guarantee.

\begin{corollary}\label[corollary]{thm:deflation_precond}
Let $\vec{P}_\mu$ be the preconditioner \cref{eqn:precond_def} corresponding to $\llbracket\vec{A}\rrbracket_r$, the rank-$r$ truncated SVD of $\vec{A}$, and $\theta\in[\lambda_d,\lambda_{r+1}]$.
Then the $t$-th PCG iterate (\cref{def:PCG}) corresponding to the preconditioner $\vec{P}_\mu$ satisfies
\begin{equation*}
    \frac{\| \vec{A}_\mu^{-1}\vec{b} - \pcg{t}(\mu) \|_{\vec{A}_\mu}}{\| \vec{A}_\mu^{-1}\vec{b}\|_{\vec{A}_\mu}}
    \leq 2 \exp\left( - \frac{2t}{\sqrt{\kappa_{r+1}(\mu)}} \right) .
\end{equation*}
\end{corollary}

When $\lambda_{r+1} \ll \lambda_1$ and $\mu$ is not too large, then the rate of convergence guaranteed by the bound can be much faster than without preconditioning ($r=0$).\footnote{Standard CG also satisfies bounds in terms of $\lambda_{r+1}/\lambda_d$, at least in exact arithmetic; see  \cref{thm:cg_condno_bd_full} and the discussion in \cref{sec:comparison}.}
Of course, any potential benefits to convergence must be weighted against the cost of constructing the spectral deflation preconditioner, and exact deflation, which requires computing the top eigenvectors of $\vec{A}$, can be costly.

\subsubsection{Nystr\"om preconditioning}
\label{sec:nystrom_precond}

Techniques from randomized numerical linear algebra allow near-optimal low-rank approximations to $\vec{A}$ to be computed very efficiently \cite{halko_martinsson_tropp_11,tropp_webber_23}, so we might hope that the corresponding preconditioner works nearly as well as spectral deflation, while being much cheaper to compute.

In particular, it's reasonable to take $\vec{U}\vec{D}\vec{U}^\T$ as the eigendecomposition of the randomized block-Krylov Nystr\"om approximation, \emph{mathematically} defined as
\begin{equation}\label{eqn:nystromBKI_def}
    \vec{A}\langle \vec{K}_s \rangle := (\vec{A}\vec{K}_s)(\vec{K}_s^\T\vec{A}\vec{K}_s)^\dagger (\vec{K}_s^\T\vec{A}),
\end{equation}
where $\vec{K}_s := [\vec{\Omega}~\vec{A}\vec{\Omega}~\cdots~\vec{A}^{s-1}\vec{\Omega}] \in \mathbb{R}^{d\times(s\ell)}$ and $\vec{\Omega}\in\mathbb{R}^{d\times \ell}$ is a matrix of independent standard normal random variables.\footnote{In this paper, we only consider Nystr\"om preconditioning where $\vec{\Omega}$ is a Gaussian matrix. Variants of Nystr\"om approximation based on subsampling rows/columns of $\vec{A}$ are effective in some settings (e.g. Kernel Ridge Regression); see \cite[\S2.2.3]{frangella_tropp_udell_23} for a discussion.}
This variant of the Nystr\"om approximation is among the most powerful randomized low-rank approximation algorithms, and can be implemented using $s$ matrix-loads.
Note, however, that the algorithm should not be implemented as written in \cref{eqn:nystromBKI_def}. 
In particular, one should avoid using a monomial basis $\vec{K}_s$ and carefully structure the interactions with $\vec{A}$ to avoid unnecessary costs; see e.g. \cite{tropp_webber_23} for pseudocode.

Nystr\"om-based preconditioning is effective in theory and practice, and has been an active area of research in recent years \cite{martinsson_tropp_20,frangella_tropp_udell_23,carson_dauzickaite_24,diaz_epperly_frangella_tropp_webber_23,zhao_xu_huang_chow_xi_24,hong_xu_hu_fessler_24,dereziski_musco_yang_25}.
Most related to the content of this paper is the theoretical analysis of \cite{frangella_tropp_udell_23}, which proves that if $s=1$, $\theta = \lambda_{\ell}(\vec{A}\langle \vec{\Omega}\rangle)$, and the sketching dimension $\ell$ is on the order of the effective dimension $d_{\textup{eff}}(\mu) := \tr\bigl(\vec{A} \vec{A}_{\mu}^{-1}\bigr)
= \sum_{i=1}^{d} {\lambda_i}/{(\lambda_i+\mu)}$,
then Nystr\"om PCG converges in at most $t = O(\log(1/\varepsilon))$ iterations; i.e. independent of any spectral properties of $\vec{A}$.
Our analysis makes use of the same general techniques as \cite{frangella_tropp_udell_23}, but is applicable when $s>1$ as well as if $\mu=0$.
We discuss the theoretical bounds of \cite{frangella_tropp_udell_23} in \cref{sec:effective_dim}.

\section{Our approach: augmented block-CG}
\label{sec:ABCG}
In this paper, we advocate for the direct application of block KSMs to \cref{eqn:regularized_ls}.
The concept of block Krylov subspaces \cref{eqn:block_krylov_def} naturally gives rise to the
block-CG algorithm \cite{oleary_80}.

\begin{definition}\label{def:BCG}
    Let $\vec{B} = [\vec{b}^{(1)}~\cdots~\vec{b}^{(m)}]$.
    The $t$-th block-CG iterates are defined as
    \begin{equation*}
    \bcg{t}{i}(\mu) := \operatornamewithlimits{argmin}_{\vec{x}\in\mathcal{K}_t(\vec{A}_\mu,\vec{B})} \| \vec{A}_\mu^{-1}\vec{b}^{(i)} - \vec{x} \|_{\vec{A}_\mu}.
    \end{equation*}
\end{definition}

The block-CG iterates $\bcg{t}{1}(\mu), \ldots, \bcg{t}{m}(\mu)$ can be simultaneously computed using $t-1$ block matrix-vector products with $\vec{A}$.
We discuss implementation and costs further in \cref{sec:runtime}.

\subsection{Implicit preconditioning}
Our main conceptual result is that by augmenting $\vec{b}$ with $\vec{\Omega}$, block-CG implicitly enjoys the benefits of certain classes of preconditioners built using $\vec{\Omega}$.
We begin with a key observation about the relation between the block Krylov subspace \cref{eqn:block_krylov_def} and the preconditioned Krylov subspace \cref{eqn:precond_krylov_def} corresponding to a certain class of preconditioners.
\begin{theorem}\label{thm:nested}
Suppose $\vec{P}_{\mu} = (\vec{I} + \vec{X})^{-1}$, where $\operatorname{range}(\vec{X}) \subseteq\mathcal{K}_{s+1}(\vec{A}_{\mu},\vec{\Omega})$.
Then,
\begin{equation*}
\mathcal{K}_t(\vec{A}_{\mu},\vec{b};\vec{P}_{\mu})
\subseteq
\mathcal{K}_t(\vec{A},\vec{b}) + \mathcal{K}_{t+s}(\vec{A},\vec{\Omega}). 
\end{equation*}
\end{theorem}

The basic idea is simple. 
By definition, $\mathcal{K}_t(\vec{A}_{\mu},\vec{b};\vec{P}_{\mu})$ consists of linear combinations of the vectors 
\begin{equation*}
    (\vec{P}_{\mu}^{-1}\vec{A}_\mu)^k \vec{P}_\mu^{-1}\vec{b} = ((\vec{I}+\vec{X}) \vec{A}_\mu)^{k} (\vec{I}+\vec{X})\vec{b},
    \quad k=0,1,\ldots, t-1,
\end{equation*}
and each $((\vec{I}+\vec{X}) \vec{A}_\mu)^{k} (\vec{I}+\vec{X})\vec{b}$ can be expressed as linear combination of vectors which live in the specified space.

\begin{proof}~
By the shift invariance of (block) Krylov subspaces, 
$\mathcal{K}_t(\vec{A},\vec{b}) + \mathcal{K}_{t+s}(\vec{A},{\Omega}) = \mathcal{K}_t(\vec{A}_\mu,\vec{b}) + \mathcal{K}_{t+s}(\vec{A}_\mu,\vec{\Omega})$.
Hence, without loss of generality it suffices to consider the case $\mu=0$.
For notational simplicity we will denote $\vec{P}_{\mu}$ by $\vec{P}$.
      
We proceed by induction, beginning with the base case $t = 1$. 
Observe that, 
\begin{equation*}
    \mathcal{K}_1(\vec{P}^{-1}\vec{A}, \vec{P}^{-1}\vec{b}) 
    = \operatorname{span}(\vec{P}^{-1}\vec{b}) 
    = \operatorname{span}((\vec{I} + \vec{X})\vec{b})  
    = \operatorname{span}(\vec{b} + \vec{X}\vec{b}).
\end{equation*}
Clearly $\vec{b}\in\mathcal{K}_1(\vec{A},\vec{b})$ and $\vec{X}\vec{b}\in \operatorname{range}(\vec{X}) \subseteq \mathcal{K}_{s+1}(\vec{A}, \vec{\Omega})$, so $\vec{\vec{b} + \vec{X}\vec{b}} \subseteq \mathcal{K}_t(\vec{A},\vec{b}) + \mathcal{K}_{t+s}(\vec{A}, \vec{\Omega})$ as desired.

Now, assume that 
\begin{equation*}
   \mathcal{K}_{t-1}(\vec{P}^{-1}\vec{A}, \vec{P}^{-1}\vec{b}) 
   \subseteq \mathcal{K}_{t-1}(\vec{A},\vec{b}) + \mathcal{K}_{(t-1)+s}(\vec{A}, \vec{\Omega}).
\end{equation*}
We consider the order $t$ subspace 
\begin{align*}
    &\mathcal{K}_t(\vec{P}^{-1}\vec{A}, \vec{P}^{-1}\vec{b}) = \operatorname{span}\{\vec{P}^{-1}\vec{b}, (\vec{P}^{-1}\vec{A})\vec{P}^{-1}\vec{b}, \ldots, 
    \nonumber\\&\hspace{13em}
    (\vec{P}^{-1}\vec{A})^{t-2}\vec{P}^{-1}\vec{b}, 
    (\vec{P}^{-1}\vec{A})^{t-1}\vec{P}^{-1}\vec{b}\}.
\end{align*}
\newpage
From the inductive hypothesis, we know that for all $j = 0, ..., t-2$,
\begin{align*}
    (\vec{P}^{-1}\vec{A})^j\vec{P}^{-1}\vec{b} &\in \mathcal{K}_{t-1}(\vec{P}^{-1}\vec{A}, \vec{P}^{-1}\vec{b}) \\
    &\subseteq \mathcal{K}_{t-1}(\vec{A},\vec{b}) + \mathcal{K}_{t+s-1}(\vec{A}, \vec{\Omega}) \\
    &\subseteq \mathcal{K}_{t}(\vec{A},\vec{b}) + \mathcal{K}_{t+s} (\vec{A}, \vec{\Omega}),
\end{align*}
where the last inclusion follows from the nested property of Krylov subspaces. 

Thus, it remains to show that $(\vec{P}^{-1}\vec{A})^{t-1}\vec{P}^{-1}\vec{b} \in \mathcal{K}_{t}(\vec{A},\vec{b}) + \mathcal{K}_{t+s}(\vec{A}, \vec{\Omega})$.
Towards this end, let $\vec{v} = (\vec{P}^{-1}\vec{A})^{t-2}\vec{P}^{-1}\vec{b}$ and observe that
\begin{equation*}
    (\vec{P}^{-1}\vec{A})^{t-1}\vec{P}^{-1}\vec{b} 
    = (\vec{P}^{-1}\vec{A})\vec{v} 
    = (\vec{I} + \vec{X})\vec{A}\vec{v} 
    = \vec{A}\vec{v} + \vec{X}\vec{A}\vec{v}.
\end{equation*}
Clearly $\vec{X}\vec{A}\vec{v}\subset \operatorname{range}(\vec{X})\subseteq \vec{K}_{s+1}(\vec{A},\vec{\Omega})$.
Moreover, as noted, $\vec{v} \in \mathcal{K}_{t-1}(\vec{A},\vec{b}) + \mathcal{K}_{t+s-1}(\vec{A}, \vec{\Omega})$ and so 
\begin{equation*}
    \vec{A}\vec{v} \in  \vec{A}\mathcal{K}_{t-1}(\vec{A},\vec{b}) + \vec{A}\mathcal{K}_{t+s-1}(\vec{A}, \vec{\Omega}) 
    \subseteq \mathcal{K}_{t}(\vec{A},\vec{b}) + \mathcal{K}_{t+s}(\vec{A}, \vec{\Omega}). 
\end{equation*}
This proves the result.
\end{proof}

\Cref{thm:nested} allows the performance of augmented block-CG to be related to PCG with a broad choice of preconditioner.
\begin{theorem}\label{thm:main}
Fix any matrix $\vec{\Omega}\in\mathbb{R}^{d\times m}$ and let $\vec{P}_{\mu} = (\vec{I} + \vec{X})^{-1}$ be any preconditioner where $\operatorname{range}(\vec{X}) \subseteq\mathcal{K}_{s+1}(\vec{A},\vec{\Omega})$.
Define the augmented starting block $\vec{B} = [\vec{b}~\vec{\Omega}]$. 
Then, for any $t\geq s$, the $t$-th \BCG{} iterate is related to the $(t-s)$-th \PCG{} iterate corresponding to the preconditioner $\vec{P}_\mu$  in that
\begin{equation*}
\| \vec{A}_\mu^{-1}\vec{b} - \bcg{t}{1}(\mu) \|_{\vec{A}_\mu} 
\leq \| \vec{A}_\mu^{-1}\vec{b} - \pcg{t-s}(\mu) \|_{\vec{A}_\mu}.
\end{equation*}
\end{theorem}

\begin{proof}~ 
Clearly $\mathcal{K}_t(\vec{A},\vec{b}) + \mathcal{K}_{t+s}(\vec{A},\vec{\Omega})
\subseteq\mathcal{K}_{t+s}(\vec{A},[\vec{b}~\vec{\Omega}])$, so the result follows immediately from \cref{thm:nested}, the optimality of block-CG, and the definition of preconditioned-CG.
\end{proof}

\begin{remark}
    \Cref{thm:main} asserts that augmented block-CG automatically performs no worse than Nystr\"om PCG (with the best choice of $s$ and $\theta$) after the same number of matrix-loads.\footnote{Recall, $s$ matrix-loads are used to compute the Nystr\"om approximation $\vec{A}\langle \vec{K}_s\rangle$, whose range lives in $\mathcal{K}_{s+1}(\vec{A},\vec{\Omega})$.}
    In particular, when $\vec{U}\vec{D}\vec{U}^\T = \vec{A}\langle \vec{K}_s\rangle$,where $\langle \vec{K}_s\rangle$ is the Nystr\"om approximation \cref{eqn:nystromBKI_def}, then the deflation preconditioner $\vec{P}_\mu$ defined in \cref{eqn:precond_def} has the form 
    \begin{equation*}
    \vec{P}^{-1} 
    = \vec{I} + \vec{X}
    , \quad\text{where
    $\operatorname{range}(\vec{X})\subseteq\mathcal{K}_{s+1}(\vec{A},\vec{\Omega})$}.
    \end{equation*}
    This is a \emph{deterministic} statement that explains the relative performance of the algorithms in \cref{fig:intro_fig}.
\end{remark}

In \cref{sec:analysis}, we use use \cref{thm:main}, and a new bound for Nystr\"om preconditioning (\cref{thm:condno_prob}) when $\vec{\Omega}$ is a Gaussian random matrix to prove probabilistic bounds for block-CG reminiscent of \cref{thm:deflation_precond} for the spectral deflation preconditioner.

\subsection{Computational costs}
\label{sec:runtime}

The block-CG method is a standard algorithm in numerical analysis, and there are many mathematically equivalent implementations; i.e. different implementations which produce the same output in exact arithmetic.
Here we describe a block-Lanczos based implementation of block-CG that is particularly suitable for our applications to ridge regression and Gaussian sampling.

The block-Lanczos algorithm applied to $(\vec{A},\vec{B})$ for $t$ iterations produces $\vec{Q}_t$ with orthonormal columns and $\vec{T}_t$ with bandwidth $2m+1$ satisfying
\begin{equation*}
    \operatorname{range}(\vec{Q}_t) =\mathcal{K}_t(\vec{A},\vec{B})
,\quad
\vec{T}_t = \vec{Q}_t^\T \vec{A}\vec{Q}_t .
\end{equation*}
A standard computation reveals that
\begin{equation}\label{eqn:bcg_alg}
    \bcg{k}{1}(\mu) = \|\vec{b}\| \vec{Q}_t (\vec{T}_t + \mu\vec{I})^{-1} \vec{e}_1,
\end{equation}
where $\vec{e}_1 = [1,0,\ldots,0]^\T$. 
In particular, since $\mathcal{K}_t(\vec{A}_\mu,\vec{B})=\mathcal{K}_t(\vec{A},\vec{B})$ for any scalar $\mu$,  \cref{eqn:bcg_alg} can be used to compute the block-CG iterate for \emph{multiple values of $\mu$} using the same $\vec{Q}_t$ and $\vec{T}_t$.
For more details on block-Lanczos and the connection to block-CG, see e.g., \cite{oleary_80,saad_80,chen_24,tichy_meurant_simonova_25}.

There are \emph{many} variants of block-Lanczos, since there are many possible orthogonlization schemes that can be used within block-Lanczos type algorithms \cite{carson_lund_rozloznik_thomas_22,nakatsukasa_tropp_24,balabanov_grigori_25}.
The following provides a high-level overview of the costs associated with block-Lanczos with full-reorthogonalization.

\begin{theorem}\label{thm:cost}
Suppose $\vec{B}\in\mathbb{R}^{d\times m}$. 
The block-Lanczos algorithm (with full orthogonalization) applied to $(\vec{A},\vec{B})$ for $t$ iterations produces $\vec{Q}_t$ and $\vec{T}_t$ using: 
\begin{itemize}
    \item $t-1$ matrix-loads of $\vec{A}$ (for a total of $m(t-1)$ total matrix-products)
    \item $O(dm^2 t^2)$ floating point operations (in addition to products with $\vec{A}$), and 
    \item $O(dmt)$ storage (in addition to storage required for $\vec{A}$).
\end{itemize}
Subsequently, for any $\mu\geq 0$, the $t$-th \BCG{} iterates $\bcg{t}{1}(\mu)$, can be computed using an additional $O(dmt + m^3t)$ floating point operations operations.
\end{theorem}

\begin{proof}~ 
Let $n = \dim(\mathcal{K}_t(\vec{A},\vec{B}))$. 
Since $\vec{B}$ has $m$ columns, by the definition \cref{eqn:block_krylov_def} of the block Krylov subspace, $n\leq \min\{d,mt\}$.
At each iteration the block-Lanczos algorithm performs one matrix-load ($m$ parallel matrix products) and lower-order arithmetic.
The dominant storage cost is storing $\vec{Q}_t$, which has $n\leq mt$ columns of length $n$.
The arithmetic cost is dominated by the $O(dn^2) = O(dm^2t^2)$ operations required to obtain $\vec{Q}_t$ using full reorthogonalization.
Recall that $\vec{T}_t$ (and hence $\vec{T}_t + \mu\vec{I}$) is a $n\times n$ matrix with bandwidth $O(m)$.
Thus, the linear system $(\vec{T}_t+\mu\vec{I})\vec{z} = \vec{e}_1$ can be solved in $O(n m^2) = O(m^3t)$ time. 
Subsequently, since $\vec{Q}_t$ is a $d \times n$ matrix, $\vec{Q}_t \vec{z}$ can be computed in $O(dn) = O(dmt)$ time.
\end{proof}

We emphasize that, similar to unpreconditioned CG (see e.g., \cite{hanke_17,frommer_maass_99,calvetti_reichel_shuibi_04}), after the block-Lanczos algorithm has been run, the cost to compute the block-CG iterates for multiple values of $\mu$ is relatively small.
This is in contrast to Nystr\"om PCG, which is is not particularly well suited for solving \cref{eqn:regularized_ls} for multiple values of $\mu$; while the Nystr\"om approximation can be reused, PCG must be re-run and new products with $\vec{A}$ computed.

Notice that the costs (besides matrix-loads) scale with the block-size $m$.
We again remind the reader that, while \cref{thm:main} guarantees augmented block-CG outperforms Nystr\"om preconditioning in terms of matrix-loads, our theory provides no guarantees for other cost matrices like matrix-products or floating point operations. 
Users should carefully consider the costs and benefits associated with the various algorithms in the context of their own computing environments.

\begin{remark}
By avoiding orthogonalizing and storing the whole Krylov basis, it is in fact possible to reduce the floating point cost to $O(dmtL + m^3t L)$ and the storage to $O(dm + d L)$, where $L$ is the number of values of $\mu$ at which one wishes to evaluate $\bcg{t}{1}(\mu)$ \cite{oleary_80}. 
However, there are several caveats to such an implementation.
First, such methods require knowing the values of $\mu$ at which one wishes to evaluate $\bcg{t}{1}(\mu)$ ahead of time. 
Second, the finite-precision behavior of such methods can differ greatly from the exact arithmetic behavior.
We explore the impacts of finite precision arithmetic  \cref{sec:numerical_appendix}, where we perform some numerical experiments that indicate that full reorthogonalization may be unnecessary in some situations.
\end{remark}

For methods based on the standard Lanczos algorithm, there is quite a bit of theory which guarantees that such methods can still work, even without orthogonalization \cite{paige_71,paige_76,druskin_knizhnerman_91,greenbaum_89,musco_musco_sidford_18,chen_24}.
Block Krylov methods must contend with additional difficulties such as blocks becoming ill-conditioned or even rank-deficient as the iteration proceeds (which can happen even in exact arithmetic).
Unfortunately, much less theory is known about the block-Lanczos algorithm in finite precision arithmetic \cite{simonova_tichy_25}. 
Without reorthogonalization, we observe that in many cases the block-CG method fails to converge at all, while in other cases it does converge for some time; see \cref{sec:numerical_appendix}. 
Understanding this behavior is well beyond the scope of the current work.



    

\section{Probabilistic bounds}
\label{sec:analysis}
In this section we prove probabilistic bounds for augmented block-CG.
In light of \cref{thm:main}, our strategy for deriving bounds for block-CG is simply to derive bounds for Nystr\"om PCG.
To do this, we follow the approach developed in \cite{frangella_tropp_udell_23} and use a deterministic bound for the condition number of the Nystr\"om preconditioned system.
Note that block-CG always performs at least as well as CG (in terms of matrix-loads), and hence is guaranteed to converge (in exact arithmetic).

The following is a minor generalization of Proposition 5.3 in \cite{frangella_tropp_udell_23} to allow arbitrary $\theta$.
The proof is contained in \cref{sec:proofs:nystrompcg} for completeness. 
\begin{proposition}\label[proposition]{thm:condno_perturb}
Let $\vec{P}_\mu$ be the Nystr\"om preconditioner \cref{eqn:precond_def} corresponding to the Nystr\"om approximation $\vec{A}\langle \vec{K} \rangle$ for any $\vec{K}$ and shift parameter $\theta\geq0$.
Then
\begin{equation*}
\kappa(\vec{P}_\mu^{-1/2}\vec{A}_\mu\vec{P}_\mu^{-1/2})
\leq \big(\theta + \mu + \|\vec{A} - \vec{A}\langle \vec{K}\rangle\|\big)\left( \frac{1}{\theta+\mu} + \frac{1}{\lambda_d + \mu} \right).
\end{equation*}
\end{proposition}
In particular, \cref{thm:condno_perturb} implies that if $\|\vec{A} - \vec{A}\langle \vec{K}\rangle\|\approx \lambda_{r+1}$ and $\theta\in[\lambda_d,\lambda_{r+1}]$, then $\kappa(\vec{P}_{\mu}^{-1/2}\vec{A}_{\mu}\vec{P}_{\mu}^{-1/2})
\approx \kappa_{r+1}(\mu)$.

Low-rank approximation is one of the most studied problems in randomized numerical linear algebra, and many bounds have been developed. 
In the remainder of this section, we employ these bounds to obtain guarantees for augmented block-CG.

Our first result, which we prove in \cref{sec:NLA}, is a ``numerical analysis'' style bound based on the error gurantees for Nystr\"om low-rank approximation in \cite{tropp_webber_23}.
\begin{theorem}\label{thm:main_rate}
Let $\vec{\Omega}\in\mathbb{R}^{d\times(r+p)q}$, where $q \geq \log(1/\delta)/\log(100)$, be a random Gaussian matrix.
Suppose that $p\geq 2$, 
\begin{equation*}
s \geq \min\left\{ 3 + \frac{\log(d)}{2},{\frac{3}{2}} + \frac{1}{4} \log\left( 4 + \frac{4r}{p-1}\sum_{i>r}\frac{\lambda_{i}^2}{\lambda_{r+1}^2} \right) \right\}.
\end{equation*}
Then, the \BCG{} iterate $\bcg{t}{1}(\mu)$  satisfies, with probability at least $1-\delta$, 
\begin{equation*}
    \Bigg\{ \forall \mu \geq 0: \frac{\| \vec{A}_\mu^{-1}\vec{b} - \bcg{t}{1}(\mu)\|_{\vec{A}_\mu}}{\| \vec{A}_\mu^{-1}\vec{b} \|_{\vec{A}_\mu}} 
    \leq 2 \exp \left( -\frac{t-s}{3\sqrt{\kappa_{r+1}(\mu)}}\right) \Bigg\}.
\end{equation*}
\end{theorem}
This implies that, after a small burn-in period of at most $O(\log(d))$ iterations, we have exponential convergence at a rate roughly $1/\sqrt{\kappa_{r+1}(\mu)}$.

Note that $\vec{A}\langle \vec{K}_s \rangle$ as defined in \cref{eqn:nystromBKI_def} can have rank as large as $\ell s$, so we might hope that we can deflate roughly $\ell s$ eigenvalues.
Recent work \cite{meyer_musco_musco_24,chen_epperly_meyer_musco_rao_25} implies this is more-or-less the case.
In \cref{sec:TCS}, we use this result to prove our second result, a ``theoretical computer science'' type bound.

\begin{theorem}\label{thm:main_rate_TCS}
Let $\vec{\Omega}\in\mathbb{R}^{\ell}$ be a Gaussian matrix and $r\geq 0$.
Then, for some
\begin{equation*}
s = O\left( \frac{r}{\ell} \log(\Delta) + \log\left( \frac{d}{\delta} \right) \right),
\quad
\Delta := \min_{i=1, \ldots, \ell\lceil r/\ell \rceil-1} \frac{\lambda_i - \lambda_{i+1}}{\lambda_1},
\end{equation*}
the \BCG{} iterate $\bcg{t}{1}(\mu)$  satisfies, with probability at least $1-\delta$, 
\begin{equation*}
    \Bigg\{ \forall \mu \geq 0: \frac{\| \vec{A}_\mu^{-1}\vec{b} - \bcg{t}{1}(\mu)\|_{\vec{A}_\mu}}{\| \vec{A}_\mu^{-1}\vec{b} \|_{\vec{A}_\mu}} 
    \leq 2 \exp \left( -\frac{t-s}{3\sqrt{\kappa_{r+1}(\mu)}}\right) \Bigg\}.
\end{equation*}
\end{theorem}
In contrast to \cref{thm:main_rate}, which requires $\ell \geq r+2$, \cref{thm:main_rate_TCS} allows any choice of $\ell$. 
The bound reveals that we get exponential convergence at a rate roughly $1/\sqrt{\kappa_{r+1}(\mu)}$ after roughly $s \approx r/\ell$ matrix-loads (and hence $r \approx \ell s$ matrix-vector products).
Note that the logarithmic dependence on eigenvalue gaps $\Delta$ is generally considered to be mild due to ``smoothed-analysis'' type effects of finite-precision arithmetic  \cite{meyer_musco_musco_24,chen_epperly_meyer_musco_rao_25}.

\begin{remark}
The fact that \cref{thm:main_rate,thm:main_rate_TCS} guarantee an accurate result for all $\mu \geq 0$ with high probability (as opposed to a single value of $\mu$) will be important in our applications.
In particular, it allows guarantees for computing the whole ridge-regression regularization path and will be necessary in the analysis of our algorithm for sampling Gaussian vectors.
\end{remark}

\begin{remark}
Bounds based on condition numbers (such as \cref{thm:main_rate}) are often pessimistic in practice. 
To determine how many iterations to run, it is more common to use some sort of a posteriori error estimate. 
For example, observing the residual $\| \vec{b} - \vec{A}_\mu \bcg{t}{1}(\mu)\|$ gives some indication of the quality of the solution. 
More advanced techniques can also be used \cite{meurant_tichy_24}.
\end{remark}

\subsection{Explicit bound}
\label{sec:NLA}
In this section, we prove \cref{thm:main_rate}.
We do not attempt to optimize constants, opting instead for simple arguments and clean theorem statements.

We begin by recalling an error guarantee for Nystr\"om low-rank approximation, that compares the Nystr\"om error $\| \vec{A} - \vec{A}\langle \vec{K}_s\rangle \|$ to the error $\|\vec{A} - \llbracket\vec{A}\rrbracket_r\| =\lambda_{r+1}$ of the best possible rank-$r$ approximation to $\vec{A}$.

\begin{theorem}[Theorem 9.1 in \cite{tropp_webber_23}]\label{thm:nystrom_LRA}
Suppose $\vec{\Omega}\in\mathbb{R}^{d\times (r+p)}$ is a random Gaussian matrix and define $\vec{K}_s := [\vec{\Omega}~\vec{A}\vec{\Omega}~\cdots~\vec{A}^{s-1}\vec{\Omega}]$.
Then, if $p\geq 2$,
\begin{equation*}
\log\Biggl( \frac{\EE\| \vec{A} - \vec{A}\langle \vec{K}_s\rangle \|^2}{\lambda_{r+1}^2} \Bigg)
\leq 
\frac{1}{8 (s-\frac{3}{2})^2} \log\Biggl( 4 + \frac{4r}{p-1}\sum_{i>r}\frac{\lambda_{i}^2}{\lambda_{r+1}^2} \Bigg)^2.
\end{equation*}
\end{theorem}
To prove our main convergence guarantee for our augmented block-CG, we derive a new bound for Nystr\"om preconditioning that may be of independent interest.
\begin{theorem}\label{thm:condno_prob}
Suppose $\vec{\Omega}\in\mathbb{R}^{d\times (r+p)}$ is a random Gaussian matrix and define $\vec{K}_s := [\vec{\Omega}~\vec{A}\vec{\Omega}~\cdots~\vec{A}^{s-1}\vec{\Omega}]$.
Let $\vec{P}_\mu$ be the Nystr\"om preconditioner \cref{eqn:precond_def} corresponding to the Nystr\"om approximation $\vec{A}\langle \vec{K}_s \rangle$ for any shift parameter $\theta\in[\lambda_d,\lambda_{r+1}]$.
Suppose $p\geq 2$ and 
\begin{equation*}
s \geq \min\left\{ 3 + \frac{\log(d)}{2}, \frac{3}{2} + \frac{1}{4} \log\left( 4 + \frac{4r}{p-1}\sum_{i>r}\frac{\lambda_{i}^2}{\lambda_{r+1}^2} \right) \right\}.
\end{equation*}
Then, with probability at least $99/100$,
\begin{equation*}
\bigg\{ \forall \mu \geq0 :\kappa(\vec{P}_{\mu}^{-1/2}\vec{A}_{\mu}\vec{P}_{\mu}^{-1/2})
\leq 28 \kappa_{r+1}(\mu) \bigg\}.
\end{equation*}
\end{theorem}
Note that our bound for Nystr\"om PCG applies even when $\mu=0$, whereas the bounds in \cite{frangella_tropp_udell_23} require $\mu$ to be sufficiently large relative to the sketching dimension $r+p$ and consider only $s=1$.

\begin{proof}~ 
The proof is a simple consequence of \cref{thm:nystrom_LRA,thm:condno_perturb}.
We begin by bounding $\|\vec{A}-\vec{A}\langle \vec{K}\rangle\|$.
We first note that the minimum is always attained by the second term.
Indeed, since $r\leq d$ and $\lambda_i/\lambda_{r+1}\leq 1$ for $i>r$, so since $p\geq2$, we see $4r/(p-1) \sum_{i>r} \lambda_i^2 / \lambda_{r+1}^2 \leq 4d^2$. 
Now, using properties of the logarithm, $\log(4+4d^2) = \log(1+d^2) + \log(4) \leq  \log(d^2) + \log(2) + \log(4) = \log(8)+2\log(d)$.
The claim then follows since $3/2+\log(8)/4 \leq 3$.

With the condition on $s$, \cref{thm:nystrom_LRA} guarantees that
\begin{equation*}
\EE\Bigl[ \| \vec{A} - \vec{A}\langle \vec{K}\rangle \|^2 \Bigr]
\leq \mathrm{e}^{1/2} \cdot \lambda_{r+1}^2.
\end{equation*}
Applying Markov's inequality, we therefore have that 
\begin{equation*}
\PP\Bigl[ \| \vec{A} - \vec{A}\langle \vec{K}\rangle \|^2 \geq 100 \cdot \mathrm{e}^{1/2} \lambda_{r+1}^2 \Bigr] \leq \frac{1}{100}.
\end{equation*}
Condition on the event that $\{\| \vec{A} - \vec{A}\langle \vec{K}\rangle \| \leq 10 \mathrm{e}^{1/4} \lambda_{r+1}\}$.
Since $\theta \leq \lambda_{r+1}$ and $\mu\geq 0$,
\begin{equation*}
\theta+\mu+\|\vec{E}\|
\leq (1+10 \mathrm{e}^{1/4}) (\lambda_{r+1}+\mu)
\leq 14(\lambda_{r+1}+\mu).
\end{equation*}
Next, since $\theta \geq \lambda_d$, 
\begin{equation*}
\left( \frac{1}{\theta+\mu} + \frac{1}{\lambda_d + \mu} \right)
\leq \frac{2}{\lambda_d+\mu}.
\end{equation*}
The result follows by combining the above equations.
\end{proof}

The proof of \cref{thm:main_rate} is now straightforward.
To get a high probability bound, we use a simple boosting ``trick''. 
We expect more refined results can be obtained directly.

\par\textit{Proof of \cref{thm:main_rate}.}
We first analyze the case $q=1$ (i.e. $\delta = 1/100$).
By \cref{thm:condno_prob}, we are guaranteed that with probability at least $99/100$,
\begin{equation}\label{eqn:kappa_small_event}
\Big\{ \forall \mu\geq 0: \kappa(\vec{P}_{\mu}^{-1/2}\vec{A}_{\mu}\vec{P}_{\mu}^{-1/2})
\leq 28 \kappa_{r+1}(\mu) \Big\}.
\end{equation}
When the event in \cref{eqn:kappa_small_event} holds, \cref{thm:pcg_condno_bd} guarantees that,
\begin{equation*}
\Bigg\{ \forall \mu\geq 0: \frac{\| \vec{A}^{-1}\vec{b} - \pcg{t}(\mu) \|_{\vec{A}_\mu}}{\|\vec{A}^{-1}\vec{b}\|_{\vec{A}_\mu}}
\leq 2\exp\left( -\frac{2t}{\sqrt{28\kappa_{r+1}(\mu)}} \right) \Bigg\}.
\end{equation*}
The result follows from \cref{thm:main} and that $\sqrt{28}/2 < 3$.

We now analyze the case for arbitrary $q\geq 1$. 
Partition $\vec{\Omega}=[\vec{\Omega}_1~\cdots~\vec{\Omega}_q]\in\mathbb{R}^{d\times (r+2)q}$.
Then, for $i=1, \ldots, q$, $\vec{\Omega}_i\in\mathbb{R}^{d\times(r+2)}$ are independent Gaussian matrices.
By our analysis of the $q=1$ case, block-CG with starting block $[\vec{b}~\vec{\Omega}_i]$ fails to reach the specified accuracy within $t$ iterations with probability at most $1/100$.
The probability that all of these $q$ (independent) instances fail to reach this accuracy is therefore at most $(1/100)^q \leq \delta$.
Finally, since 
\begin{equation*}
\mathcal{K}_t(\vec{A},[\vec{b}~\vec{\Omega}_i])
\subseteq \mathcal{K}_t(\vec{A},\vec{\Omega})
\end{equation*}
block-CG with starting block $[\vec{b}~\vec{\Omega}]$ performs no worse than block-CG with starting block $[\vec{b}~\vec{\Omega}_i]$ (for any $i$), and hence fails to reach the stated accuracy within $t$ iterations with probability at most $\delta$.
\endproof

\subsection{Improved matrix-vector product complexity}
\label{sec:TCS}
In this section, we prove \cref{thm:main_rate_TCS}.
The general approach is identical to the proof of \cref{thm:main_rate}, but we use a different bound for Nystr\"om low-rank approximation.

\begin{theorem}[Theorem 1.3 in \cite{chen_epperly_meyer_musco_rao_25}]
\label{thm:any_block_size}
Let $\vec{\Omega}\in\mathbb{R}^{d\times\ell}$ be a Gaussian matrix.
Then, for some
\begin{equation*}
    s = O\left( \frac{r}{\ell\sqrt{\varepsilon}} \log(\Delta) + \log\left(\frac{d}{\delta \varepsilon} \right) \right),
    \quad
    \Delta := \min_{i=1, \ldots, \ell\lceil r/\ell \rceil-1} \frac{\lambda_i - \lambda_{i+1}}{\lambda_1},
\end{equation*}
with probability at least $1-\delta$, there is a matrix $\vec{Q}\in\mathbb{R}^{d\times r}$ with orthonormal columns and $\operatorname{range}(\vec{Q}) \subseteq \mathcal{K}_{s}(\vec{A},\vec{\Omega})$ such that
\begin{equation*}
    \| \vec{A} - \vec{Q}\vec{Q}^\T\vec{A} \| < (1+\varepsilon) \lambda_{r+1}.
\end{equation*}
\end{theorem}

Note that \cref{thm:any_block_size} is not for Nystr\"om low-rank approximation. 
However, Nystr\"om low-rank approximation is better than projection-based low-rank approximation \cite{tropp_webber_23}, which allows us to prove \cref{thm:main_rate_TCS}.
\par\textit{Proof of \cref{thm:main_rate_TCS}.}
    We set $\varepsilon = 10 \mathrm{e}^{1/4}-1$.
    Since $\operatorname{range}(\vec{Q})\subseteq \operatorname{range}(\vec{K}_s)$, the standard monotonicity property of the Nystr\"om approximation implies $\| \vec{A} - \vec{A}\langle \vec{K}_s \rangle \| \leq \| \vec{A} - \vec{A}\langle \vec{Q} \rangle \|$.
    Next, \cite[Lemma 5.2]{tropp_webber_23} asserts that for any matrix $\vec{Q}$ with orthonormal columns, $\|\vec{A} - \vec{A}\langle \vec{Q}\rangle \| \leq \|\vec{A} - \vec{Q}\vec{Q}^\T\vec{A}\|$.
    Thus, by \cref{thm:any_block_size} (with $\varepsilon$ set to an appropriate constant), the given condition on $s$ ensures that, with probability at least $1-\delta$, $\| \vec{A} - \vec{A}\langle \vec{K}_s \rangle \| < 10 \mathrm{e}^{1/4} \lambda_{r+1}$.
    The proof then follows similarly to the proof of \cref{thm:main_rate}.
\endproof

\begin{remark}
A number of algorithms for Gaussian process regression make use of block-KSMs to simultaneously solve a linear system and apply a matrix-function to a collection of Gaussian random vectors (for stochastic trace estimation) \cite{gardner_pleiss_weinberger_bindel_wilson_18,xu_cai_huang_chow_xi_25}.
Our bounds are relevant to this setting, and provide further theoretical justification to the use of block-KSMs in these settings.
\end{remark}

\section{Gaussian sampling}
\label{sec:gaussian_sample}

Several applications in data science and statistics require sampling Gaussians with a given mean and covariance \cite{thompson_33,besag_york_molli_91,bardsley_12,wang_jegelka_17}.
A standard approach is to transform an isotropic Gaussian vector.
Indeed, suppose $\vec{b}\sim\mathcal{N}(\vec{0},\vec{I})$ (i.e. that entries of $\vec{b}$ are independent standard Gaussians).
Then, for positive definite $\vec{A}$,
\begin{equation*}
\bm{\mu} + \vec{A}^{1/2}\vec{b} \sim \mathcal{N}(\bm{\mu},\vec{A}).
\end{equation*}
In other words, in order to sample Gaussian vectors with covariance $\vec{A}$, it suffices to apply the matrix-square root of $\vec{A}$ to a standard Gaussian vector and then shift this result by $\bm{\mu}$. 
The most computationally difficult part of this is applying the square root of $\vec{A}$ to $\vec{b}$.

Recall the relation:
\begin{equation*}
    \vec{A}^{1/2}\vec{b}
    = \frac{2}{\pi}\int_{0}^\infty \vec{A}(\vec{A}+z^2\vec{I})^{-1}\vec{b}\, \d z.
\end{equation*}
When $t$ is sufficiently large, we might hope that $\cg{t}(z^2) \approx (\vec{A}+z^2\vec{I})^{-1}\vec{b}$.
This motivates the following definition.
\begin{definition}\label{def:lansq}
The $t$-th Lanczos square root iteration is defined as
\begin{equation*}
\lansq{t} := \frac{2}{\pi}\int_{0}^\infty \vec{A}\cg{t}(z^2)\, \d z.
\end{equation*}
\end{definition}
This and closely related methods appear throughout the literature  \cite{chow_saad_14,pleiss_jankowiak_eriksson_damle_gardner_20,chen_24}.

Often we wish to sample multiple vectors from $\mathcal{N}(\bm{\mu},\vec{A})$, and so we might use a block-variant of \cref{def:lansq}.
\begin{definition}\label{def:blansq}
Let $\vec{B} = [\vec{b}^{(1)}~\cdots~\vec{b}^{(m)}]$.
The $t$-th block-Lanczos square root iterate is defined as
\begin{equation*}
\blansq{t}{i} := \frac{2}{\pi}\int_{0}^\infty \vec{A}\bcg{t}{i}(z^2)\, \d z.
\end{equation*}
\end{definition}
Both algorithms can be efficiently implemented using the (block)-Lanczos algorithms; see e.g. \cite{chen_24}.

We can use the bounds for augmented block-CG from \cref{sec:analysis} to derive bounds for the block-Lanczos square root algorithm.
In particular, using \cref{thm:main_rate}, we prove:
\begin{theorem}\label{thm:lansq}
    Let $\vec{b}^{(1)}, \ldots, \vec{b}^{(m)}\in\mathbb{R}^{d}$ be independent standard Gaussian random vectors.
    Suppose 
    \[
    r \leq  (m-1) \left\lceil\frac{\log(m/\delta)}{\log(100)}\right\rceil^{-1} - 2,
    \]
    Then, with probability at least $1-\delta$,
    \[
    \Bigg\{ \forall i: \frac{\| \vec{A}^{1/2}\vec{b}^{(i)} - \blansq{t}{i} \|}{\|\vec{A}^{1/2}\|\| \vec{b}^{(i)} \|} 
    \leq \log\left(16\kappa\right) \exp\left( -\frac{t - (2 + \log(d)/2)}{3\sqrt{\kappa_{r+1}(0)}} \right)  \Bigg\}.
    \]
\end{theorem}

Similarly, using \cref{thm:main_rate_TCS}, we prove:
\begin{theorem}\label{thm:lansq_TCS}
    Let $\vec{b}^{(1)}, \ldots, \vec{b}^{(m)}\in\mathbb{R}^{d}$ be independent standard Gaussian random vectors and $r\geq 0$.
    Then, for some
    \begin{equation*}
    s = O\left( \frac{r}{m} \log(\Delta) + \log\left( \frac{d}{\delta} \right) \right),
    \quad
    \Delta := \min_{i=1, \ldots, (m-1)\lceil r/(m-1) \rceil-1} \frac{\lambda_i - \lambda_{i+1}}{\lambda_1},
    \end{equation*}
     with probability at least $1-\delta$,
    \[
    \Bigg\{ \forall i: \frac{\| \vec{A}^{1/2}\vec{b}^{(i)} - \blansq{t}{i} \|}{\|\vec{A}^{1/2}\|\| \vec{b}^{(i)} \|} 
    \leq \log\left(16\kappa\right) \exp\left( -\frac{t - s}{3\sqrt{\kappa_{r+1}(0)}} \right)  \Bigg\}.
    \]
\end{theorem}

\par\textit{Proof of \cref{thm:lansq,thm:lansq_TCS}.}
Our proof is optimized for simplicity rather than sharpness.
For notational convenience, let $\vec{b} = \vec{b}^{(1)}$ and $\vec{\Omega} = [\vec{b}^{(2)}, \ldots, \vec{b}^{(m)}]$.
As noted in \cref{sec:gaussian_sample}, $\vec{\Omega}\in\mathbb{R}^{d\times(m-1)}$ is independent of $\vec{b}$.
Consider the event 
\begin{equation}\label{eqn:bcg_eta_bd}
    \Bigg\{ \forall \mu\geq 0: \frac{\| \vec{A}_{\mu}^{-1}\vec{b} - \bcg{t}{1}(\mu) \|_{\vec{A}_{\mu}}}{\|\vec{A}_{\mu}^{-1} \vec{b}\|_{\vec{A}_\mu} }
    \leq \varepsilon_t'(\mu) \Bigg\},
\end{equation}
where 
\begin{equation*}
\varepsilon_t'(\mu)
=
2\exp\left( -\frac{t - s}{3\sqrt{\kappa_{r+1}(\mu)}} \right).
\end{equation*}

To prove \cref{thm:lansq}, note that the choice of $r$ ensures that $m-1 \geq (r+2)q$ for $q = \lceil\log(m/\delta)/\log(100)\rceil$.
Therefore, with $s = 2 +\log(d)/2$, \cref{thm:main_rate} guarantees that \cref{eqn:bcg_eta_bd} holds with probability at least $1-\delta/m$.
Likewise, to prove \cref{thm:lansq_TCS}, note that the \cref{thm:main_rate_TCS} guarantees that \cref{eqn:bcg_eta_bd} holds with probability at least $1-\delta/m$ (where we have used that $m\leq d$ so that $\log(dm/\delta) = O(\log(d/\delta))$).

From this point on we condition on \cref{eqn:bcg_eta_bd}, and the two proofs are identical.
Applying standard norm inequalities we obtain a bound
\begin{align*}
    \| \vec{A}\vec{A}_{\mu}^{-1}\vec{b} - \vec{A}\bcg{t}{1}(\mu) \|
    &=\| \vec{A}\vec{A}_{\mu}^{-1/2}\vec{A}_\mu^{1/2}(\vec{A}_\mu^{-1}\vec{b} - \bcg{t}{1}(\mu) )\|
    \\&\leq\| \vec{A}\vec{A}_{\mu}^{-1/2}\|\| \vec{A}_{\mu}^{-1}\vec{b} - \bcg{t}{1}(\mu)\|_{\vec{A}_\mu}.
\end{align*}
Under the assumption the event in \cref{eqn:bcg_eta_bd} holds, we bound
\begin{equation*}
    \|\vec{A}_{\mu}^{-1}\vec{b} - \bcg{t}{1}(\mu) \|_{\vec{A}_{\mu}}
    \leq  \varepsilon_t(\mu) \|\vec{A}_{\mu}^{-1}\vec{b}\|_{\vec{A}_{\mu}}
    \leq \varepsilon_t(\mu) \|\vec{A}_{\mu}^{-1/2}\|\|\vec{b}\|.
\end{equation*}
We observe that
\begin{equation*}
    \|\vec{A}\vec{A}_\mu^{-1/2}\|
    = \max_i \frac{\lambda_i}{\sqrt{\lambda_i+\mu}} 
    = \frac{\lambda_1}{\sqrt{\lambda_1+\mu}}
\end{equation*}
and similarly 
\begin{equation*}
\|\vec{A}_{\mu}^{-1/2}\|
= \max_i \frac{1}{\sqrt{\lambda_i+\mu}}
= \frac{1}{\sqrt{\lambda_d+\mu}}.
\end{equation*}
Therefore, since $\varepsilon_t'(\mu)\leq \varepsilon_t'(0)$
\begin{equation}\label{eqn:integrand_bd}
    \| \vec{A}\vec{A}_{\mu}^{-1}\vec{b} - \vec{A}\bcg{t}{1}(\mu) \|
    \leq \frac{\lambda_1 \varepsilon_t'(0) \|\vec{b}\|}{\sqrt{(\lambda_1+\mu)(\lambda_d+\mu)}}.
\end{equation}
Therefore, applying the triangle inequality for integrals and \cref{eqn:integrand_bd} we compute a bound
\begin{align}
    \| \vec{A}^{1/2}\vec{b} - \blansq{t}{1} \|
    \nonumber&= \frac{2}{\pi} \biggl\| \int_0^\infty \vec{A}(\vec{A}+z^2\vec{I})^{-1}\vec{b} - \vec{A}\bcg{t}{1}(z^2) \d{z} \biggr\|
    \\\nonumber&\leq \frac{2}{\pi} \int_0^\infty \| \vec{A}(\vec{A}+z^2\vec{I})^{-1}\vec{b} - \vec{A}\bcg{t}{1}(z^2) \| \d{z}
    \\&\leq \frac{2}{\pi} \int_0^\infty \frac{\lambda_1 \varepsilon_t'(0) \|\vec{b}\|}{\sqrt{(\lambda_1+z^2)(\lambda_d+z^2)}} \d{z}.
    \label{eqn:lansq_bd_final}
\end{align}
A direct computation reveals
\begin{equation}\label{eqn:ellipticK_eval}
     \int_0^\infty \frac{1}{\sqrt{(\lambda_1+z^2)(\lambda_d+z^2)}} \d{z} = \frac{K(1-\lambda_1/\lambda_d)}{\sqrt{\lambda_d}},
\end{equation}
where $K(m) := \int_0^{\pi/2}(1-m\sin^2(z))^{-1/2}\d{z}$ is the complete elliptic integral of the first kind.
Standard bounds on elliptic integrals (see \cref{thm:elliptic_k_bd}) guarantee that, for all $x> 1$, 
\begin{equation}\label{eqn:ellipticK_bd}
    \frac{2}{\pi}\sqrt{x} K(1-x)\leq \frac{5}{4\pi}\log(16x)
    \leq \frac{1}{2} \log(16x).
\end{equation}
Therefore, applying $\cref{eqn:ellipticK_bd}$ with $x = \kappa = \lambda_1/\lambda_d$ to \cref{eqn:lansq_bd_final,eqn:ellipticK_eval} we get a bound
\begin{equation*}
\| \vec{A}^{1/2}\vec{b} - \blansq{t}{1} \|
\leq \frac{\sqrt{\lambda_1}}{2} \log\left(16\kappa\right) \varepsilon_t'(0) \|\vec{b}\|.
\end{equation*}
Since $\sqrt{\lambda_1} = \|\vec{A}^{1/2}\|$, this is the desired result for $\vec{b}^{(1)}$. 

To obtain the final result, we observe that $\vec{b}^{(1)}, \ldots, \vec{b}^{(m)}$ and $\blansq{t}{1}, \ldots, \blansq{t}{m}$ are permutation invariant in distribution.
Thus, the above result in fact applies to each $\vec{b}^{(i)}$ and $\blansq{t}{i}$ pair with probability at least $1-\delta/m$. 
Applying a union bound over these $m$ events gives the main result.
\endproof

\Cref{thm:lansq,thm:lansq_TCS} gives bounds on the matrix-vector products required by the block-Lanczos method. For instance, \cref{thm:lansq} implies we can use roughly
\begin{equation*}
m\Big(\log(d) + \sqrt{\kappa_{r+1}}  \log\big(\log(\kappa)/\varepsilon\big)\Big)
\end{equation*}
matrix-vector products, where $r = O(m/\log(m))$.
In contrast, existing bounds for methods like the single-vector Lanczos square root method \cref{def:lansq}
\cite{chen_greenbaum_musco_musco_22,pleiss_jankowiak_eriksson_damle_gardner_20} sample a single Gaussian vector with roughly $\sqrt{\kappa} \log(1/\varepsilon)$ matrix-vector products.
Therefore, when $\lambda_{r+1} \ll \lambda_1$, the total number of matrix-vector products is reduced significantly by using the block method in \cref{def:blansq}.

\begin{remark}
For sampling Gaussians, some preconditioning-like type approaches can be applied used \cite{chow_saad_14,pleiss_jankowiak_eriksson_damle_gardner_20,frommer_ramirezhidalgo_schweitzer_tsolakis_24}.
As far as we are aware, there are no theoretical guarantees for such methods similar to the ones presented in the current work.
In addition, analogous to the case of linear systems, we expect block methods to enjoy the benefits of working over a larger subspace.
\end{remark}

\begin{remark}
Relating KSMs for matrix functions to KSMs for (shifted) linear systems via integral relations widely used in order to design and analyze algorithms for matrix functions; see e.g. \cite{frommer_guttel_schweitzer_14,schweitzer_25,chen_greenbaum_musco_musco_22,xu_chen_24}.
However, it is generally difficult to use preconditioners for linear systems for matrix functions, as preconditioned Krylov subspaces do not generally satisfy a shift invariance property.
As such, the integral relation no longer yields an efficient to implement algorithm. 
On the other hand, we expect similar analyses, based on implicit preconditioning with block-CG, to work for other functions.
\end{remark}

\subsection{Sampling with the inverse covariance}

Computing $\vec{A}^{-1/2}\vec{b}$ is used to transform vectors with covariance matrix $\vec{A}$ to ``whitened'' coordinate space, where the the covariance is the identity, and has found use in a number of data-science applications
\cite{kuss_rasmussen_05,pleiss_jankowiak_eriksson_damle_gardner_20}.

We can define an approximation for applying inverse square roots similar to the block-Lanczos square root iterate (\cref{def:blansq}).

\begin{definition}\label{def:blanisq}
Let $\vec{B} = [\vec{b}^{(1)}~\cdots~\vec{b}^{(m)}]$.
The $t$-th block-Lanczos inverse square root iteration is defined as
\begin{equation*}
\blanisq{t}{i} := \frac{2}{\pi}\int_{0}^\infty \bcg{t}{i}(z^2)\, \d z.
\end{equation*}
\end{definition}
Note that $\blanisq{t}{i}
= \vec{A}^{-1}\blansq{t}{i}$.
Therefore, \cref{thm:lansq} immediately gives a bound for \cref{def:blanisq} since
\begin{align*}
\| \vec{A}^{-1/2}\vec{B} - \bLanisq{t}\|
&=\| \vec{A}^{-1/2}\vec{B} - \vec{A}^{-1}\bLansq{t} \|
\\&= \| \vec{A}^{-1} ( \vec{A}^{1/2}\vec{B} - \bLansq{t} ) \|
\\&\leq \|\vec{A}^{-1}\| \| \vec{A}^{1/2}\vec{B} - \bLansq{t} \|.
\end{align*}

\section{Discussion and comparison with past work}
\label{sec:past_work}

There are may other related approaches to solving the problems in \cref{sec:intro}.
For instance, KSMs are commonly used to solve Tikahonov Regression problems, using the shift-invariance of Krylov subspaces \cref{eqn:shift_invariance} solve for to whole regularization path efficiently \cite{hanke_17,frommer_maass_99,calvetti_reichel_shuibi_04}.
Similarly, as discussed in \cref{sec:gaussian_sample},  KSMs are also for tasks involving matrix functions, including sampling Gaussian vectors \cite{chow_saad_14,pleiss_jankowiak_eriksson_damle_gardner_20,chen_24}.
The techniques discussed in \cref{sec:gaussian_sample} are likely applicable other functions as well.

The block-CG algorithm is widely used to solve linear systems with multiple right hand sides, and in such settings working over the block Krylov subspace is often advantageous compared to working over the individual Krylov subspaces \cite{oleary_80,meurant_94,feng_owen_peric_95,birk_frommer_13}.
Block-CG is also used to solve single linear systems. 
For instance, so-called enlarged KSMs split the right hand size vector $\vec{b}$ into multiple vectors using a domain decomposition approach, and then apply block-CG to this collection of vectors
\cite{grigori_moufawad_nataf_16,grigori_tissot_17,grigori_tissot_19}.
Augmenting a KSM with random vectors was shown to be beneficial in numerical experiments appearing in concurrent work \cite[Appendix B]{zimmerling_druskin_simoncini_25}.

While we are unaware of any bounds similar to \cref{thm:main_rate_TCS}, bounds similar to \cref{thm:main_rate} are known.
In particular, \cite[\S4]{oleary_80} derives a convergence bound for block-CG terms of $\sqrt{\lambda_{1}/\lambda_{d-m+1}}$; i.e. the condition number of $\vec{A}$ with the \emph{bottom} $m$ eigenvalues removed.\footnote{Note that the ordering of the eigenvalues in \cite{oleary_80} is reversed from the present paper.}
At a high level, \cite{oleary_80} obtains this bound by showing that the block Krylov subspace contains vectors which can approximately annihilate a subset of the eigenvalues of $\vec{A}$.
Thus, in principle, other sets of eigenvalues (besides the bottom ones) can be eliminated.
The bounds of \cite{oleary_80} work for any starting block (but as a result, are more technical, depending on how ``good'' the starting block is).
It would not be surprising such quantities can be bounded when the starting block is Gaussian.

KSMs are widely used for both low-rank approximation and problems involving matrix-functions, but their behavior on these problems is fundamentally different. 
As such, it would be interesting to try to avoid relying on a black-box reduction to low-rank approximation. 
For example, while low-rank approximation is complicated by the presence of repeated eigenvalues, solving linear systems is not.
The approach of \cite{oleary_80} is closely related to the approaches used in \cite{tropp_webber_23,musco_musco_15} to prove bounds for low-rank approximation, so it is conceivable that, by taking the approach of \cite{oleary_80}, but using modern tools from randomized numerical linear algebra, one may be able to directly derive better bounds for randomized versions of block-CG.


More broadly, there are many randomized algorithms for linear systems and regression. 
For instance, a recent line of work shows that certain stochastic iterative methods enjoy an implicit preconditioning-like effect \cite{derezinski_lejeune_needell_rebrova_24,lok_rebrova_24,derezinski_rebrova_24,dereziski_yang_24,derezinski_needell_rebrova_yang_25}, with convergence independent of the top eigenvalues. 
These algorithms have runtime guarantees better than KSMs in some settings \cite{derezinski_needell_rebrova_yang_25}.
In addition, for tall least squares problems, the sketch-and-precondition framework for building a preconditioner for KSMs is extremely effective \cite{rokhlin_tygert_08,martinsson_tropp_20}, resulting in algorithms with optimal theoretical complexity \cite{clarkson_woodruff_13,chenakkod_derezinski_dong_rudelson_24}.
A number of recent works in randomized numerical linear algebra
\cite{chen_hallman_23,tropp_webber_23,meyer_musco_musco_24,chen_epperly_meyer_musco_rao_25}
have made use of various nestedness properties of block-KSMs similar in flavor to \cref{thm:nested} in order to obtain more efficient algorithms.

\section{Numerical experiments}
\label{sec:numerical}

We implement the block-CG and the block-Lanczos square root iterates using the block-Lanczos algorithm with full-reorthogonalization in Python.
Code to replicate the experiments in this paper is available at \url{https://github.com/tchen-research/precond_without_precond}.
The $s$ matrix-loads required to build the Nystr\"om preconditioner are accounted for in the plots.
For Nystr\"om PCG, we set $\theta$ as the smallest eigenvalue of the Nystr\"om approximation.
We use the same random Gaussian matrix $\vec{\Omega}\in\mathbb{R}^{d\times \ell}$ for Nystr\"om PCG and block-CG. 
More numerical experiments, including without full-reorthogonalization, are included in \cref{sec:comparison,sec:numerical_appendix}.

\subsection{Convergence}
\label{sec:numerical:convergence}

We begin by comparing the convergence of the methods discussed in this paper on several test problems.
The results of this experiment are illustrated in \cref{fig:iter}, and as expected, our augmented block-CG method outperforms the other methods, often by orders-of-magnitude. 
We also observe that Nystr\"om PCG benefits from using $s>1$; i.e. from performing more than one passes over $\vec{A}$ when building the preconditioner.

As discussed in \cref{sec:comp_assm}, throughout this paper we assume that matrix-loads 
(iterations) are the dominant cost. 
For reference, we have also show convergence as a function of matrix-vector products in \cref{fig:iter_matvec}.

\begin{figure}[ht]
    \centering
    \hfil\includegraphics[scale=0.6,trim={0 0 .7cm 0},clip]{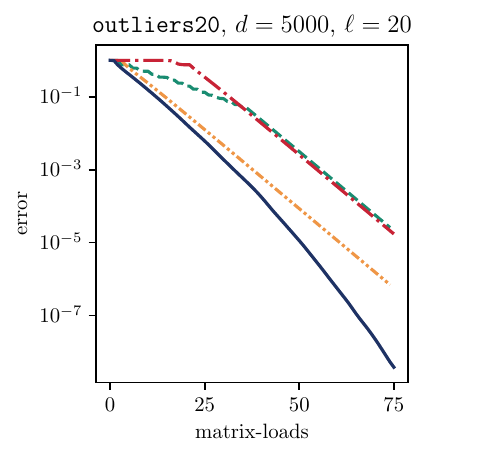}
    \hfil\includegraphics[scale=0.6,trim={.6cm 0 .7cm 0},clip]{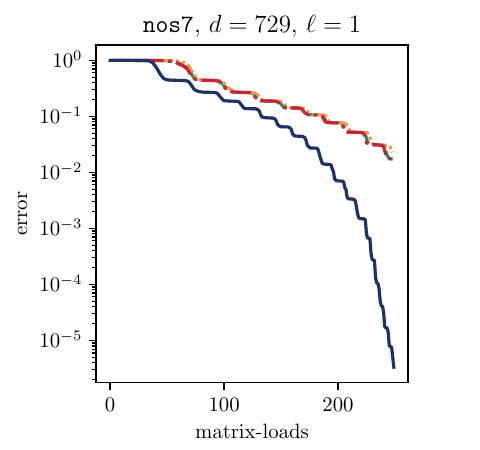}
    \hfil\includegraphics[scale=0.6,trim={.6cm 0 .7cm 0},clip]{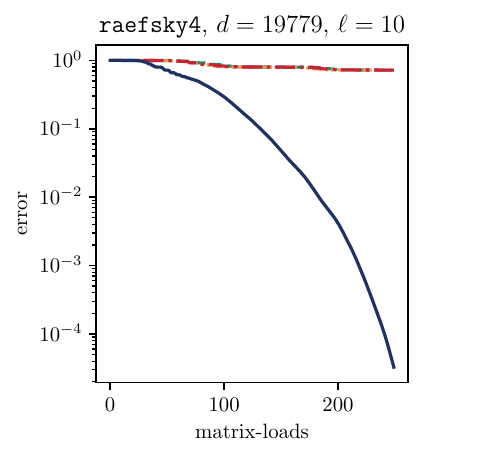}
    \caption{
    Relative error $\|\vec{A}^{-1}\vec{b} - \mathsf{alg}\|_{\vec{A}} / \|\vec{A}^{-1}\vec{b}\|_{\vec{A}}$ versus matrix-loads for
    block-CG 
    ({\protect\raisebox{0mm}{\protect\includegraphics[scale=0.68]{imgs/legend/solid.pdf}}}), 
    CG 
    ({\protect\raisebox{0mm}{\protect\includegraphics[scale=0.68]{imgs/legend/dashdot.pdf}}}), and
    Nystr\"om PCG with
    $s=1$ 
    ({\protect\raisebox{0mm}{\protect\includegraphics[scale=0.68]{imgs/legend/dash.pdf}}})
    and $s=3$
    ({\protect\raisebox{0mm}{\protect\includegraphics[scale=0.68]{imgs/legend/dashdotdot.pdf}}}) on several test problems.}
    \label{fig:iter}
\end{figure}

\subsubsection{Intro figure}

\Cref{fig:intro_fig} is run on the \texttt{fastdecay} problem ($d=8000$) with $\ell = 10$.
Nystr\"om preconditioning is run with a grid of parameters.
Specifically, we use depths $s\in\{1,5,11\}$ and a range of $\theta$, including the choice $\theta = \lambda_{\ell}(\vec{A}\langle \vec{\Omega}\rangle)$.

The timings are produced based on a setting where $\vec{A}$ cannot be stored in its entirety in fast memory. 
In particular, we store $\vec{A}$ in 8 separate $1000\times 8000$ chunks and perform (block) matrix-vector products with $\vec{A}$ by sequentially loading a single chunk from the disk into random access memory and performing the appropriate part of the products.
Similar to \cite[\S6.1]{tropp_webber_23}, we observe that the runtime of the algorithm is dominated by the cost of loading these chunks into random access memory.

\subsection{Regularization parameter}

In \cref{fig:ridge_regression} we plot the error after a fixed number of matrix loads as a function of $\mu$.
As expected, this plot indicates that our augmented block-CG outperforms the other methods for each value of $\mu$.
Perhaps more importantly, our block-CG method can efficiently compute the solution to \cref{eqn:regularized_ls} for many values of $\mu$, \emph{without the need for additional matrix-loads} (see \cref{thm:cost}).
This is in contrast to Nystr\"om PCG which requires a new run for each value of $\mu$. 
For tasks such as ridge regression, this is of note.

We note the maximum attainable accuracy of block-CG (an effect of performing the numerical experiments in finite precision arithmetic which prevents convergence to the true solution) seems to be higher than that of other methods, causing stagnation.

\begin{figure}[h]
    \centering
    \vspace{-.6em}
    \includegraphics[scale=0.6]{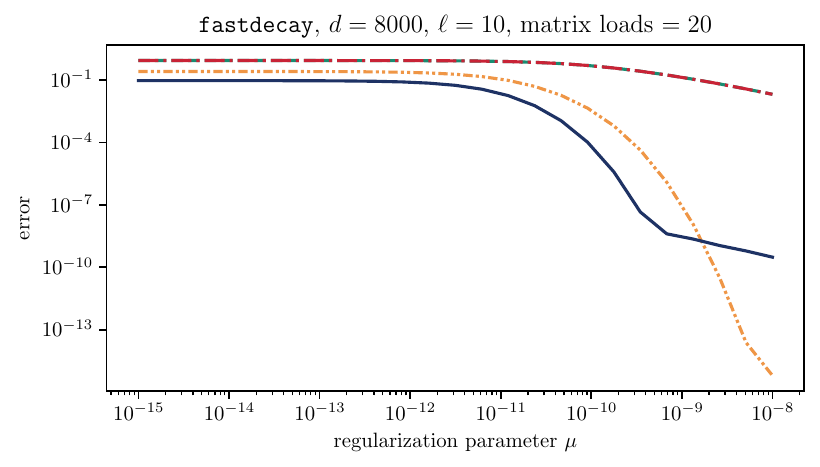}
    \vspace{-.45em}
    \caption{Relative error $\|\vec{A}_\mu^{-1}\vec{b} - \mathsf{alg}\|_{\vec{A}_\mu} / \|\vec{A}_\mu^{-1}\vec{b}\|_{\vec{A}_\mu}$ after a fixed number of matrix-loads as a function of the regularization parameter $\mu$ for
    block-CG 
    ({\protect\raisebox{0mm}{\protect\includegraphics[scale=0.68]{imgs/legend/solid.pdf}}}), 
    CG 
    ({\protect\raisebox{0mm}{\protect\includegraphics[scale=0.68]{imgs/legend/dashdot.pdf}}}), and
    Nystr\"om PCG with
    $s=1$ 
    ({\protect\raisebox{0mm}{\protect\includegraphics[scale=0.68]{imgs/legend/dash.pdf}}})
    and $s=3$
    ({\protect\raisebox{0mm}{\protect\includegraphics[scale=0.68]{imgs/legend/dashdotdot.pdf}}}).
    }
    \vspace{-1em}
    \label{fig:ridge_regression}
\end{figure}

\subsection{Sampling Gaussians}

We perform an experiment comparing the block-Lanczos square root iterate to the standard Lanczos square root iterate.
In particular, in \cref{fig:gaussian}, we compare the maximum error in approximating $\vec{A}^{1/2}\vec{b}^{(i)}$, $i=1, \ldots, m$  for the two methods.
All runs of the single-vector method are done in parallel so that the number of matrix-loads is independent of $m$.
The block method outperforms the single-vector method, and the performance gains are often significant.

\begin{figure}[ht]
    \centering
    \hfil\includegraphics[scale=0.6,trim={0 0 .7cm 0},clip]{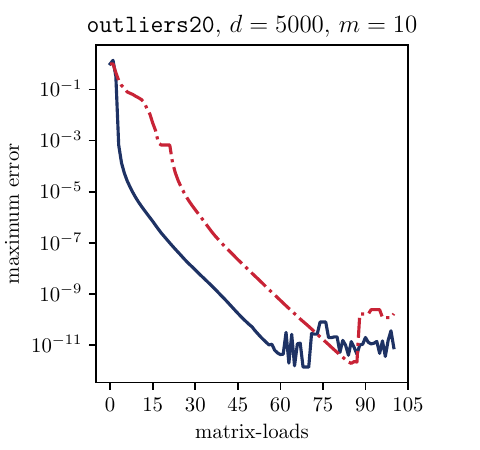}
    \hfil\includegraphics[scale=0.6,trim={.6cm 0 .7cm 0},clip]{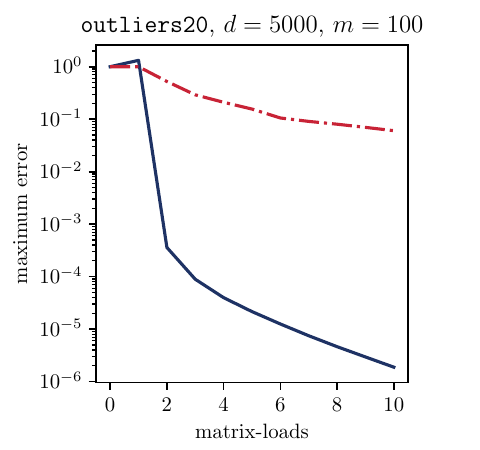}
    \hfil\includegraphics[scale=0.6,trim={.6cm 0 .7cm 0},clip]{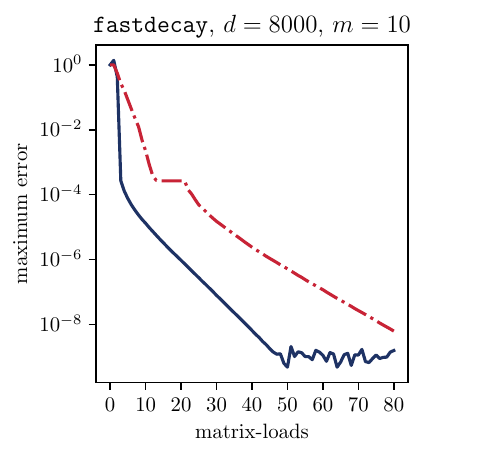}
    \caption{Maximum relative sample error $\max_i \|\vec{A}^{1/2}\vec{b}_i - \mathsf{alg}\| / \|\vec{A}^{1/2}\vec{b}_i\|$ versus matrix-loads for Lanczos square root
    ({\protect\raisebox{0mm}{\protect\includegraphics[scale=0.68]{imgs/legend/dashdot.pdf}}})
    and block-Lanczos square root
    ({\protect\raisebox{0mm}{\protect\includegraphics[scale=0.68]{imgs/legend/solid.pdf}}}).}
    \label{fig:gaussian}
\end{figure}

\section{Outlook}

We have introduced a variant of the block-CG method which outperforms Nystr\"om PCG in certain settings.
Our work provides theoretical evidence of the virtues of block Krylov subspace methods for solving a single linear system of equations and sampling Gaussian vectors.
We believe future work should study a more practical variant of augmented block-CG, where deflation is used to reduce the block-size as soon as it is recognized that the benefits of a large block size are no longer justified.



It would also be interesting to understand the extent to which similar ideas can be extended to other tasks involving matrix-functions, such as estimating the trace of matrix-functions.
Existing algorithms often involve applying matrix-functions to many Gaussian vectors \cite{ubaru_chen_saad_17,epperly_tropp_webber_24,meyer_musco_musco_24}, so it seems likely that using block-Krylov methods may be advantageous to using single-vector methods. 
However, existing theory does not reflect this. 



\vfill

\section*{Acknowledgements}

We thank Ethan Epperly, Daniel Kressner, and David Persson for useful discussions and feedback.

\section*{Disclaimer}

This paper was prepared for informational purposes by the Global Technology Applied Research center of JPMorgan Chase \& Co. This paper is not a merchandisable/sellable product of the Research Department of JPMorgan Chase \& Co. or its affiliates. Neither JPMorgan Chase \& Co. nor any of its affiliates makes any explicit or implied representation or warranty and none of them accept any liability in connection with this paper, including, without limitation, with respect to the completeness, accuracy, or reliability of the information contained herein and the potential legal, compliance, tax, or accounting effects thereof. This document is not intended as investment research or investment advice, or as a recommendation, offer, or solicitation for the purchase or sale of any security, financial instrument, financial product or service, or to be used in any way for evaluating the merits of participating in any transaction.

\appendix

\section{More numerical experiments}
\label{sec:numerical_appendix}
In this section we provide additional numerical experiments which provide additional insight into the behavior of our augmented block-CG as well as Nystr\"om PCG.

\subsection{Convergence}

We provide more test problems comparing the convergence of block-CG with Nystr\"om PCG and CG as described in \cref{sec:numerical:convergence}.
As observed in \cref{fig:iter}, block-CG outperforms the other methods in terms of matrix-loads, and often significantly so.

Note that on the \texttt{outliers20} problem we now append $\ell = 10$ or $\ell=22$ Gaussian vectors (rather than $\ell=20$ as shown in \cref{fig:iter}). 
This problem has 20 large eigenvalues, and the small amount of oversampling significantly improves the convergence of Nystr\"om PCG with $s=1$.
We also include a test with the \texttt{bottom20} problem, which has 20 small eigenvalues.

\begin{figure}[h]
    \centering
    \hfil\includegraphics[scale=0.6,trim={0 0 .7cm 0},clip]{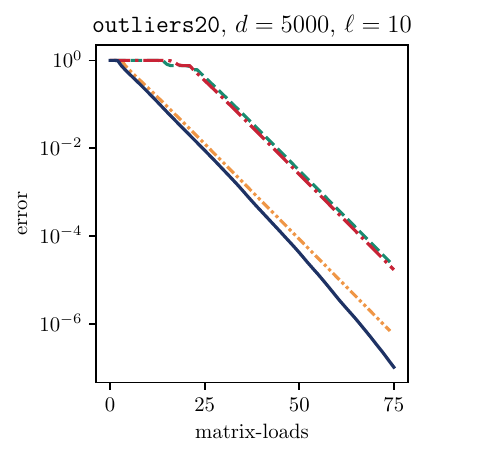}
    \hfil\includegraphics[scale=0.6,trim={.6cm 0 .7cm 0},clip]{imgs/iter_error_outliers20_l20.pdf}
    \hfil\includegraphics[scale=0.6,trim={.6cm 0 .7cm 0},clip]{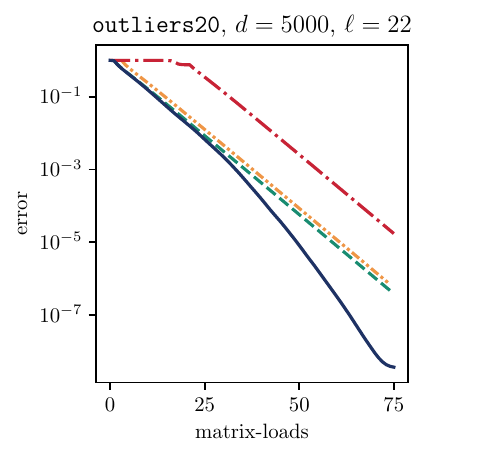}
    \\    
    \hfil\includegraphics[scale=0.6,trim={0 0 .7cm 0},clip]{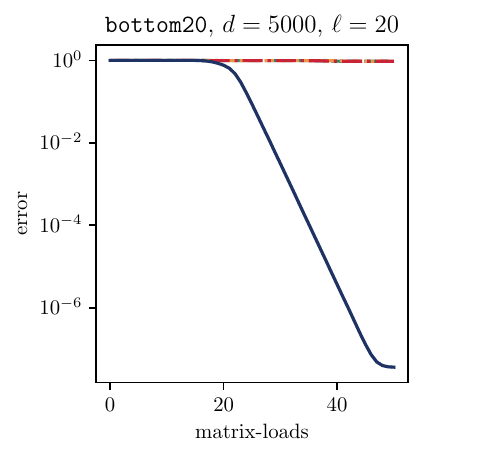}
    \hfil\includegraphics[scale=0.6,trim={.6cm 0 .7cm 0},clip]{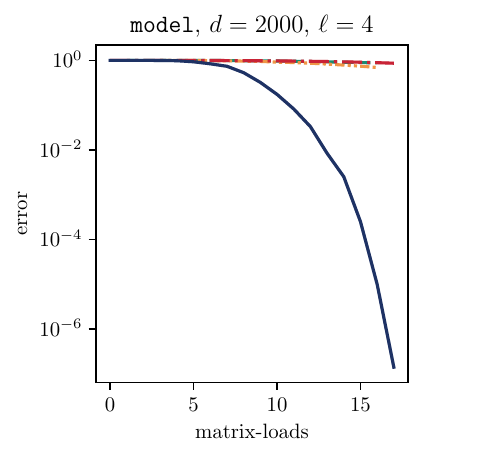}
    \hfil\includegraphics[scale=0.6,trim={.6cm 0 .7cm 0},clip]{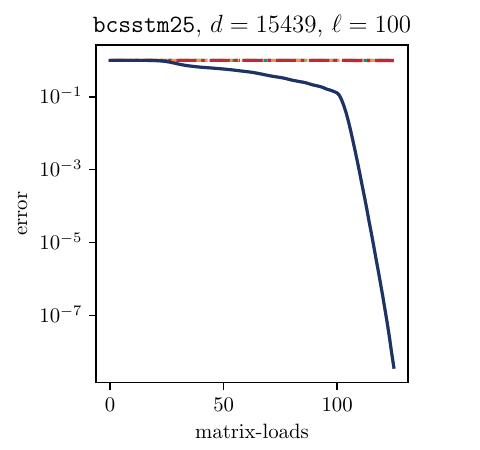}
    \caption{
    Relative error $\|\vec{A}^{-1}\vec{b} - \mathsf{alg}\|_{\vec{A}} / \|\vec{A}^{-1}\vec{b}\|_{\vec{A}}$ versus matrix-loads for
    block-CG 
    ({\protect\raisebox{0mm}{\protect\includegraphics[scale=0.68]{imgs/legend/solid.pdf}}}), 
    CG 
    ({\protect\raisebox{0mm}{\protect\includegraphics[scale=0.68]{imgs/legend/dashdot.pdf}}}), and
    Nystr\"om PCG with
    $s=1$ 
    ({\protect\raisebox{0mm}{\protect\includegraphics[scale=0.68]{imgs/legend/dash.pdf}}})
    and $s=3$
    ({\protect\raisebox{0mm}{\protect\includegraphics[scale=0.68]{imgs/legend/dashdotdot.pdf}}}) on several test problems; see also \cref{fig:iter}.
    }
    \label{fig:iter_appendix}
\end{figure}

\subsection{Block size}

\Cref{thm:main_rate_TCS} suggests the accuracy of low-rank approximations to $\vec{A}$ built from the information in the block Krylov subspace $\mathcal{K}_t(\vec{A},\vec{\Omega})$, where $\vec{\Omega}$ is a random Gaussian matrix, depends on $bk$, regardless of the individual value of $b$ and $k$.
In \cref{fig:blocksize_error}, we show the error (after a fixed number of matrix-loads) as a function of the block size, and in \cref{fig:blocksize_condno} we show the condition number of $\vec{P}_\mu^{-1/2}\vec{A}_\mu\vec{P}_\mu^{-1/2}$ as a function of the block size. 
In some cases Nystr\"om PCG with depth $s=3$ behaves similarly to Nystr\"om PCG with depth $s=1$ when the same number of matrix-vector products are used. 
For instance, on the \texttt{outliers20} problem, both methods have a significant drop in the error (due to a significant drop in the condition number) when dimension of the Krylov subspace is roughly equal to $20$, the number of outlying eigenvalues.  

\begin{figure}[H]
    \centering
    \hfil\includegraphics[scale=0.6,trim={0 0 .7cm 0},clip]{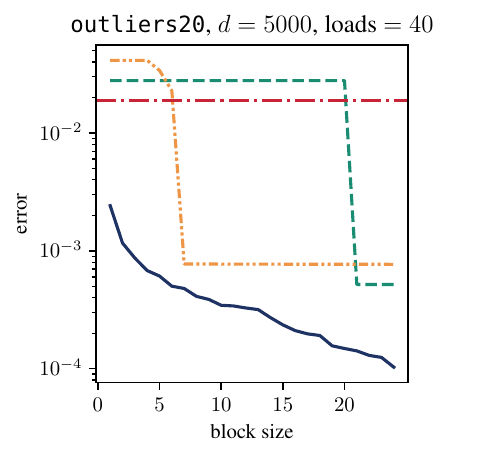}
    \hfil\includegraphics[scale=0.6,trim={.6cm 0 .7cm 0},clip]{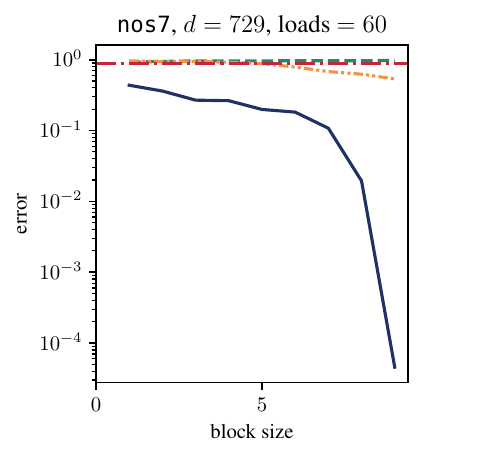}
    \hfil\includegraphics[scale=0.6,trim={.6cm 0 .7cm 0},clip]{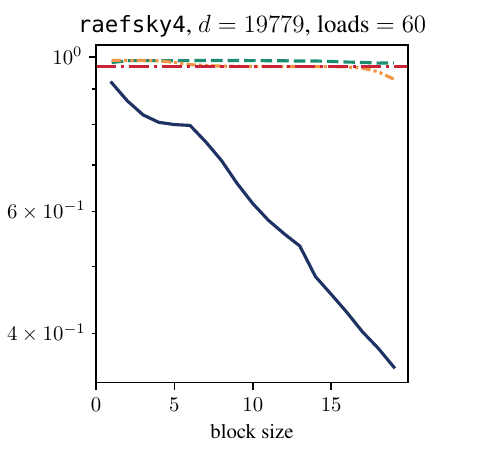}
    \caption{Relative error $\|\vec{A}^{-1}\vec{b} - \mathsf{alg}\|_{\vec{A}} / \|\vec{A}^{-1}\vec{b}\|_{\vec{A}}$ versus the Nystr\"om starting block size $\ell$ (number of columns in $\vec{\Omega}$) for
    block-CG 
    ({\protect\raisebox{0mm}{\protect\includegraphics[scale=0.68]{imgs/legend/solid.pdf}}}), 
    CG 
    ({\protect\raisebox{0mm}{\protect\includegraphics[scale=0.68]{imgs/legend/dashdot.pdf}}}), and
    Nystr\"om PCG with
    $s=1$ 
    ({\protect\raisebox{0mm}{\protect\includegraphics[scale=0.68]{imgs/legend/dash.pdf}}})
    and $s=3$
    ({\protect\raisebox{0mm}{\protect\includegraphics[scale=0.68]{imgs/legend/dashdotdot.pdf}}}) on several test problems.}
    \label{fig:blocksize_error}
\end{figure}

\begin{figure}[h]
    \centering
    \hfil\includegraphics[scale=0.6,trim={0 0 .7cm 0},clip]{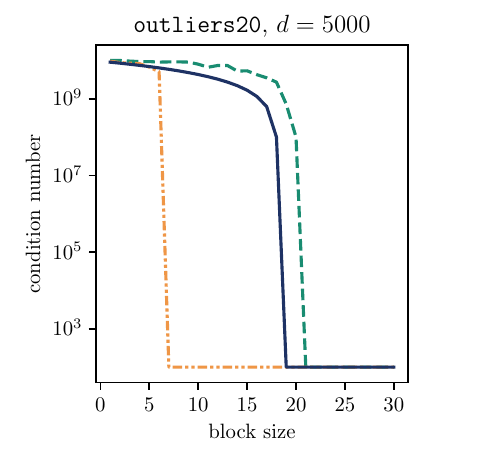}
    \hfil\includegraphics[scale=0.6,trim={.7cm 0 .7cm 0},clip]{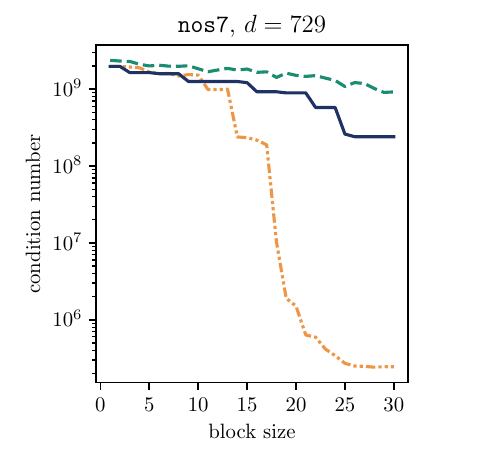}
    \hfil\hspace{-.2cm}\includegraphics[scale=0.6,trim={0cm 0 .7cm 0},clip]{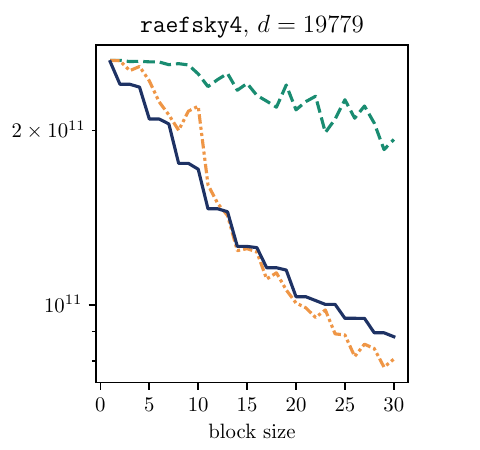}
    \caption{Condition number 
    $\kappa(\vec{P}_{\mu}^{-1/2}\vec{A}_{\mu}\vec{P}_{\mu}^{-1/2})$ of the Nystr\"om preconditioner for $s=1$ 
    ({\protect\raisebox{0mm}{\protect\includegraphics[scale=0.68]{imgs/legend/dash.pdf}}})
    and $s=3$
    ({\protect\raisebox{0mm}{\protect\includegraphics[scale=0.68]{imgs/legend/dashdotdot.pdf}}}) as a function of the block size $\ell$ (number of columns in $\vec{\Omega}$).
    For reference, we also show $(\lambda_{\ell+1}+\mu)/(\lambda_d+\mu)$ ({\protect\raisebox{0mm}{\protect\includegraphics[scale=0.68]{imgs/legend/solid.pdf}}}).
    }
    \label{fig:blocksize_condno}
\end{figure}

\subsection{Matrix-vector cost}

In \cref{fig:iter_matvec} we show convergence as a function of matrix-vector products on the same test problems as in \cref{fig:iter}.
Here our block-CG underperforms the other methods.
This does not conflict with any of our theory, which is in terms of matrix-loads.
However, it serves as a reminder that our method may not be suitable in all computational settings; see \cref{sec:comp_assm}.

\begin{figure}[ht]
    \centering
    \hfil\includegraphics[scale=0.6,trim={0 0 .7cm 0},clip]{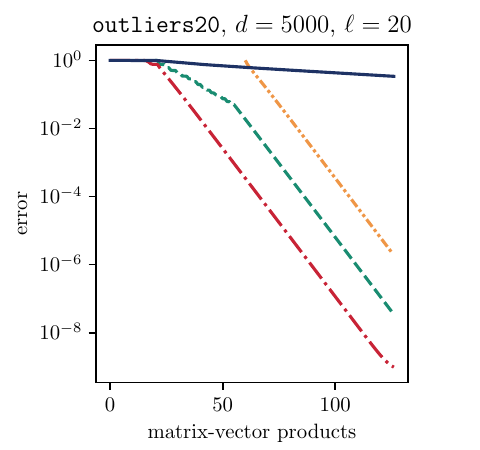}
    \hfil\includegraphics[scale=0.6,trim={.6cm 0 .7cm 0},clip]{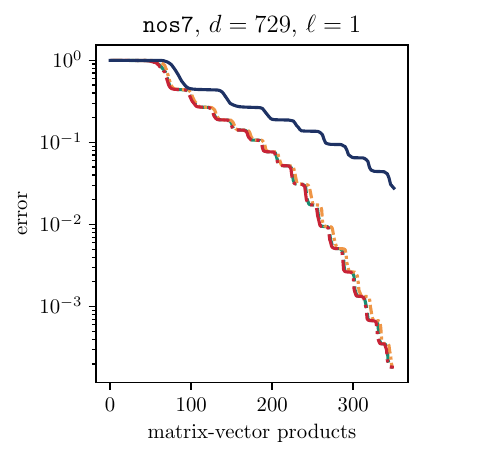}
    \hfil\includegraphics[scale=0.6,trim={.6cm 0 .7cm 0},clip]{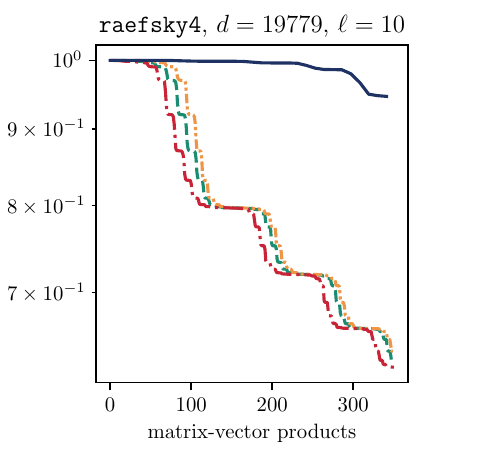}
    \\
    \hfil\includegraphics[scale=0.6,trim={0 0 .7cm 0},clip]{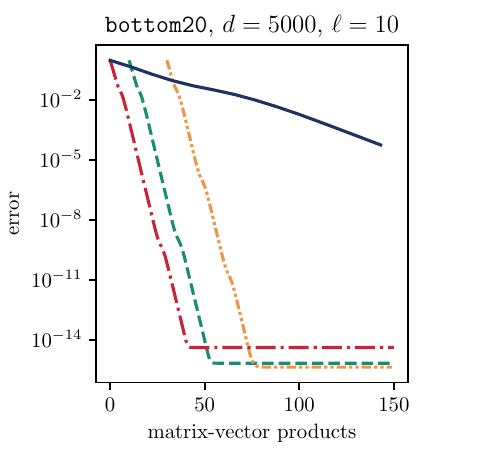}
    \hfil\includegraphics[scale=0.6,trim={.6cm 0 .7cm 0},clip]{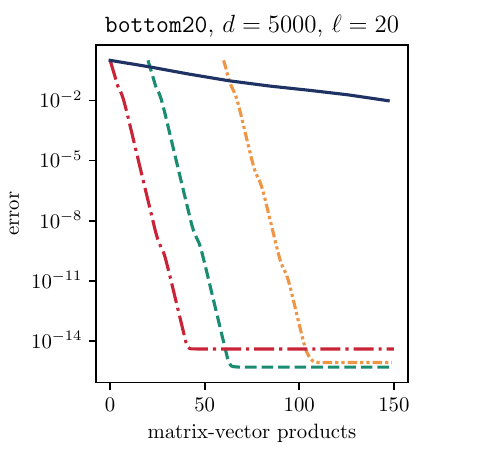}
    \hfil\includegraphics[scale=0.6,trim={.6cm 0 .7cm 0},clip]{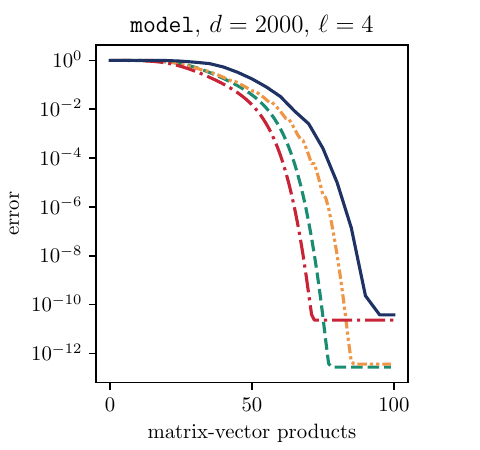}
    \caption{
    Relative error $\|\vec{A}^{-1}\vec{b} - \mathsf{alg}\|_{\vec{A}} / \|\vec{A}^{-1}\vec{b}\|_{\vec{A}}$ versus matrix-vector products for
    block-CG 
    ({\protect\raisebox{0mm}{\protect\includegraphics[scale=0.68]{imgs/legend/solid.pdf}}}), 
    CG 
    ({\protect\raisebox{0mm}{\protect\includegraphics[scale=0.68]{imgs/legend/dashdot.pdf}}}), and
    Nystr\"om PCG with
    $s=1$ 
    ({\protect\raisebox{0mm}{\protect\includegraphics[scale=0.68]{imgs/legend/dash.pdf}}})
    and $s=3$
    ({\protect\raisebox{0mm}{\protect\includegraphics[scale=0.68]{imgs/legend/dashdotdot.pdf}}}) on the same test problems as \cref{fig:iter}.
     }
    \label{fig:iter_matvec}
\end{figure}

\subsection{Test problems}

In \cref{fig:spec} we show the spectrums of the test problems we used.
\begin{figure}[H]
    \centering
    \hfil\includegraphics[scale=.6,trim={0 0 .65cm 0},clip]{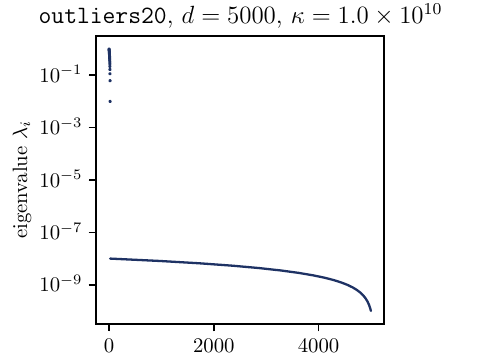}
    \hfil\includegraphics[scale=.6,trim={.55cm 0 .65cm 0},clip]{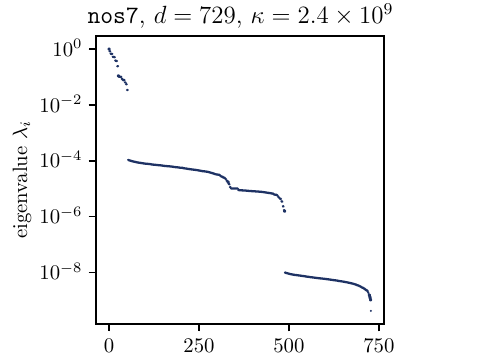}
    \hfil\includegraphics[scale=.6,trim={.55cm 0 .65cm 0},clip]{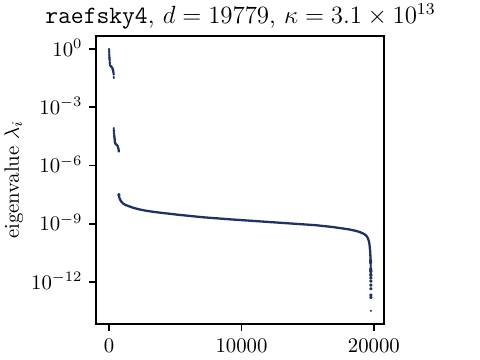}
    \\
    \hfil\includegraphics[scale=.6,trim={0 0 .65cm 0},clip]{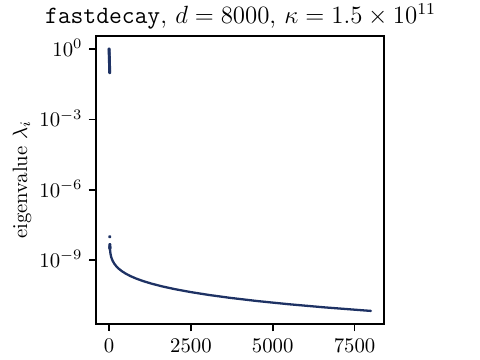}
    \hfil\includegraphics[scale=.6,trim={.55cm 0 .65cm 0},clip]{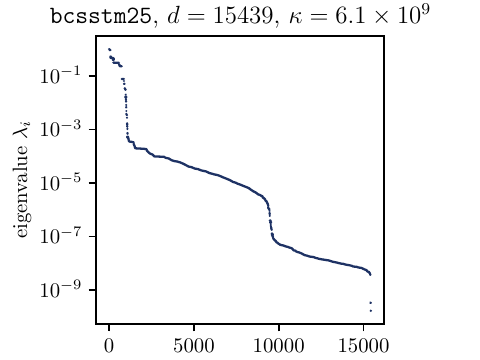}
    \hfil\includegraphics[scale=.6,trim={.55cm 0 .65cm 0},clip]{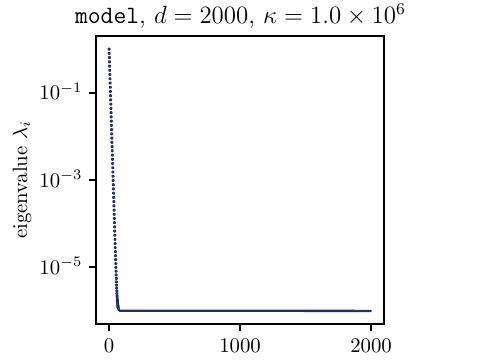}
    \\
    \includegraphics[scale=.6,trim={0 0 .65cm 0},clip]{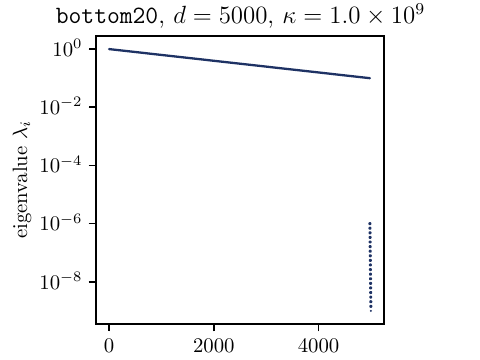}\hfill
    \caption{Spectrum of test problems used in this paper.}
    \label{fig:spec}
\end{figure}

\section{Further discussion on Nystr\"om PCG/CG/BCG}
\label{sec:comparison}
Current theory for Nystr\"om PCG and CG raises some interesting questions about the efficacy of Nystr\"om PCG if costs are measured in terms of matrix-vector products.
In particular, the bounds for Nystr\"om PCG do not really guarantee that the algorithm uses any fewer matrix-vector products than CG.
In this section we provide a discussion with the aim of raising some considerations about Nystr\"om PCG that are likely of importance to practitioners and may be interesting directions for future work.

\subsection{Bounds for CG in terms of the deflated condition number}
In \cref{thm:condno_prob} we show that if $s=O(\log(d))$, then $\kappa(\vec{P}_{\mu}^{-1/2}\vec{A}_{\mu}\vec{P}_{\mu}^{-1/2}) = O(\lambda_{r+1}/\lambda_d)$. 
It is informative to compare such a bound against similar bounds for CG. 
Towards this end, we recall a standard bound for CG; see e.g.  \cite{greenbaum_97}.
\begin{theorem}\label{thm:cg_condno_bd_full}
For any $r\geq 0$ the \CG{} iterate satisfies
\begin{equation*}
    \frac{\| \vec{A}_\mu^{-1}\vec{b} - \cg{t}(\mu) \|_{\vec{A}_\mu}}{\| \vec{A}_\mu^{-1}\vec{b} \|_{\vec{A}_\mu}}
    \leq 2\exp\left(-\frac{2(t-r)}{\sqrt{\kappa_{r+1}(\mu)}}\right).
\end{equation*}
\end{theorem}
\Cref{thm:cg_condno_bd_full} guarantees that CG converges at a rate of $\sqrt{\lambda_{r+1}/\lambda_d}$ after a burn-in period of $r$ iterations. 
This is reminiscent of the bounds for Nystr\"om PCG presented in this paper, which guarantee convergence at this rate after a preconditioner build time of roughly $O(r)$ matrix-products.

\subsection{Bounds in terms of the effective dimension}
\label{sec:effective_dim}

The main theoretical bounds for Nystr\"om PCG from \cite{frangella_tropp_udell_23} are in terms of the effective dimension 
\begin{equation*}
    d_{\textup{eff}}(\mu) := \tr\bigl(\vec{A} \vec{A}_{\mu}^{-1}\bigr)
= \sum_{i=1}^{d} \frac{\lambda_i}{\lambda_i+\mu}.
\end{equation*}
The main result is a guarantee on the the condition number of the Nystr\"om preconditioned system.

\begin{theorem}[Theorem 1.1 in \cite{frangella_tropp_udell_23}]
\label{thm:constant_condno}
Let $\vec{\Omega}\in\mathbb{R}^{d\times \ell}$ be a matrix of independent standard normal random variables for some $\ell \geq 2\lceil 1.5 d_{\textup{eff}}(\mu)\rceil + 1$.
Then, with $\theta = \lambda_{\ell}(\vec{A}\langle \vec{\Omega}\rangle)$, the Nystr\"om preconditioned system satisfies
\begin{equation}
\EE\Bigl[
\kappa(\vec{P}_{\mu}^{-1/2}\vec{A}_{\mu}\vec{P}_{\mu}^{-1/2})
\Bigr] < 28.
\end{equation}
\end{theorem}

When $\kappa(\vec{P}_{\mu}^{-1/2}\vec{A}_{\mu}\vec{P}_{\mu}^{-1/2})$ is bounded by a constant, preconditioned CG will converge to a constant accuracy (e.g. $10^{-4}$ or $10^{-16}$) in a number of iterations \emph{independent} of the condition number or spectral properties of $\vec{A}$.

We can convert \cref{thm:cg_condno_bd_full} to a bound in terms of the effective dimension by noting that all but the first $O(d_{\textup{eff}})$ eigenvalues of $\vec{A}$ are bounded by $\mu$.
In particular, \cite[Lemma 5.4]{frangella_tropp_udell_23} implies that for any $r$ and $\mu\geq 0$,
\begin{equation}
\label{eqn:deff_condno}
     r > 2d_{\textup{eff}}(\mu)
\quad\Longrightarrow\quad
\lambda_{r+1} \leq \mu.
\end{equation}
Together, \cref{thm:cg_condno_bd_full,eqn:deff_condno} give a bound for CG in terms of the effective dimension.
\begin{corollary}\label[corollary]{thm:CG_deff}
The \CG{} iterate satisfies
\begin{equation*}
    \frac{\| \vec{A}_\mu^{-1}\vec{b} - \cg{t}\|_{\vec{A}_\mu}}{\| \vec{A}_\mu^{-1}\vec{b} - \cg{0}\|_{\vec{A}_\mu}}
    \leq 2\exp\left(-\sqrt{2}(t-2 d_{\textup{eff}}(\mu))\right). 
\end{equation*}
\end{corollary}

\Cref{thm:CG_deff} asserts that after a burn in of $2 d_{\textup{eff}}(\mu)$ iterations, CG converges in $O(\log(1/\varepsilon))$ iterations. 
Thus, the theory in \cite{frangella_tropp_udell_23} does not guarantee Nystr\"om PCG has any major benefits over CG in terms of matrix-vector products.

\subsection{Is Nystr\"om PCG a good idea?}
\label{sec:nystrom_goodidea}

The existence of a bound like \cref{thm:cg_condno_bd_full} seems damning; CG automatically satisfies bounds similar to those of of Nystr\"om PCG, without the need to construct and store a preconditioner.
However, as we now discuss, the full story is much more subtle, and we believe Nystr\"om PCG is still a viable method in many instances.

First, Nystr\"om PCG is able to parallelize matrix-vector products used to build the preconditioner.
Thus, the number of matrix-loads can be considerably less than required by CG.
This is reflected in our experiments in \cref{fig:iter}.

However, these same plots indicated that in terms of matrix-products, CG significantly outperforms Nystr\"om PCG (the $s\ell$ matrix-products to build the Nystr\"om preconditioner require only one matrix-load). 
On the other hand, in apparent contradiction to our experiments, the experiments in \cite{frangella_tropp_udell_23} indicate that Nystr\"om PCG consistently and significantly outperforms CG.

Recall that the experiments in \cref{sec:numerical} are done using full-reorthogonalization (so that exact arithmetic theory is still applicable).
On the other hand, the experiments in \cite{frangella_tropp_udell_23} appear to have been done without reorthogonalzation.
It's well-known that \cref{thm:cg_condno_bd_full} does not hold finite precision arithmetic without reorthogonalization for any $r>0$.\footnote{A weaker version does still hold \cite{greenbaum_89}.} 
On the other hand, \cref{thm:pcg_condno_bd} can be expected to hold to close degree \cite{greenbaum_89,meurant_06}.
Since Nystr\"om PCG works by explicitly decreasing the condition number of $\kappa(\vec{P}_{\mu}^{-1/2}\vec{A}_{\mu}\vec{P}_{\mu}^{-1/2})$, \emph{the convergence of Nystr\"om PCG in finite precision arithmetic is still well described by the exact arithmetic theory for Nystr\"om PCG.} 

\begin{figure}[h]
    \centering
    \begin{subfigure}[t]{0.48\textwidth}\centering
    \includegraphics[scale=0.6]{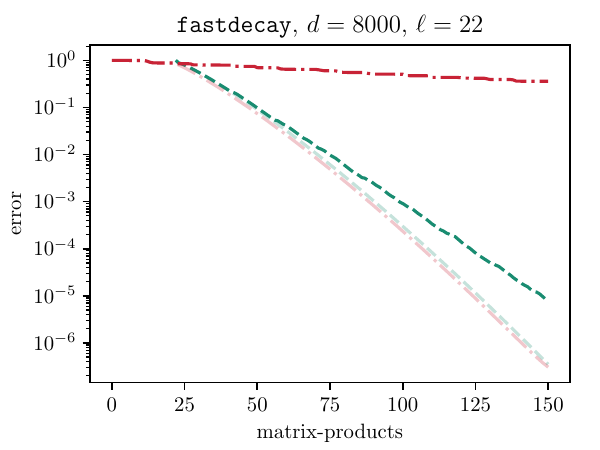}
    \vspace{-.45em}
    \caption{    CG 
    ({\protect\raisebox{0mm}{\protect\includegraphics[scale=0.68]{imgs/legend/dashdot.pdf}}}) and
    Nystr\"om PCG with
    $s=1$ 
    ({\protect\raisebox{0mm}{\protect\includegraphics[scale=0.68]{imgs/legend/dash.pdf}}}) without any reorthgonalization.
    }
    \label{fig:iter_error_fpcompare}
    \end{subfigure}
    \hfill
    \begin{subfigure}[t]{0.48\textwidth}\centering
    \includegraphics[scale=0.6]{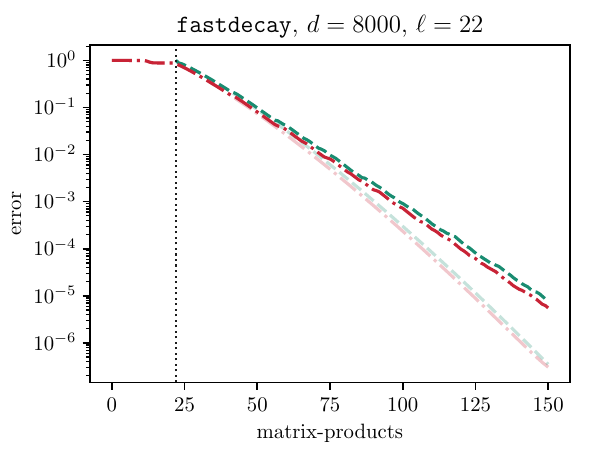}
    \vspace{-.45em}
    \caption{ 
    CG 
    ({\protect\raisebox{0mm}{\protect\includegraphics[scale=0.68]{imgs/legend/dashdot.pdf}}}) with reorthogonalization for $22$ iterations and
    Nystr\"om PCG with
    $s=1$ 
    ({\protect\raisebox{0mm}{\protect\includegraphics[scale=0.68]{imgs/legend/dash.pdf}}}) without any reorthgonalization.    
    }
    \label{fig:iter_error_fpcompare_pro}
\end{subfigure}
\caption{Relative error $\|\vec{A}^{-1}\vec{b} - \mathsf{alg}\|_{\vec{A}} / \|\vec{A}^{-1}\vec{b}\|_{\vec{A}}$ as a function of matrix products.
Light curves show convergence with full-reorthgonalization.
}
\end{figure}

We illustrate this point of comparison explicitly in \cref{fig:iter_error_fpcompare}. 
We run CG and Nystr\"om PCG on the \texttt{fastdecay} problem (for which $\lambda_1/\lambda_d \gg \lambda_{20}/\lambda_d$).
As expected, CG with reorthogonalization burns in for roughly 20 iterations, and then begins converging much more rapidly at a rate depending on $\lambda_{21}/\lambda_d$, outperforming Nystr\"om PCG with reorthogonalization ($s=1, \ell=22$), which uses 22 matrix-vector products\footnote{We need some oversampling to capture the top eigenspace.} to build the preconditioner, at which point also converges at a rate depending on $\lambda_{21}/\lambda_d$.
However, if reorthogonalization is not used, CG converges at a much slower rate depending on  $\lambda_1/\lambda_d$ while Nystr\"om PCG still converges at the faster rate depending on $\lambda_{21}/\lambda_d$.

\subsection{What if CG uses some reorthogonalization?}
\label{sec:CG_pro}

Comparing vanilla CG with Nystr\"om PCG as in \cref{fig:iter_error_fpcompare} is perhaps a bit unfair to CG; Nystr\"om PCG incurs costs in building, storing, and applying the preconditioner.
To put the algorithms on more equal footing, we might allow CG to use a similar amount of storage and arithmetic operations to do some form of reorthogonalization.
While there are many complicated schemes based on detailed analyses of the Lanczos algorithm in finite precision arithmetic \cite{parlett_scott_79,simon_84}, a simple scheme is to perform re-orthogonlization only for the first $s\ell$ iterations, and then to continue orthogonalize against these vectors in subsequent iterations. 
In this case the cost of reorthogonalization during the first $s\ell$ iterations will be $O(d s^2\ell^2)$ operations (the same as the cost to build the the Nystr\"om preconditioner) and subsequently orthogonalizing against these vectors will require $O(ds\ell)$ operations per iteration, the same as the cost to apply the Nystr\"om preconditioner to a vector.

In \cref{fig:iter_error_fpcompare_pro} we show the behavior of CG with full-reorthogonalization for the first 22 iterations and then orthogonalizing subsequent Krylov basis vectors against these first thirty vectors.
Compared to \cref{fig:iter_error_fpcompare}, we observe that CG now converges as rapidly as Nystr\"om PCG.

\subsection{block-CG}

In \cref{fig:iter} we compare the performance of these methods when we do not use full-reorthogonalzation.
In particular, we run Nystr\"om PCG with $s=3$ without any reorthogonalization. 
Building the preconditioner requires orthogonalizing $s\ell$ vectors. 
For block-CG and CG use full reorthogonalization until we have a set of $3 \ell$
vectors and then continue to orthogonalize only against these vectors.
For reference, we also display the convergence of block-CG and CG if no reorthogonalization is used at all.
Unfortunately, far less is known about the behavior of block-CG and block-Lanczos in finite precision arithmetic than is known about their single-vector counterparts.

\begin{figure}[h!]
    \centering
    \hfil\includegraphics[scale=0.6]{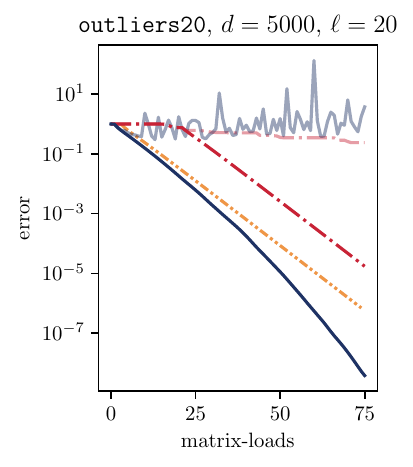}
    \hfil\includegraphics[scale=0.6,trim={.52cm 0 0 0},clip]{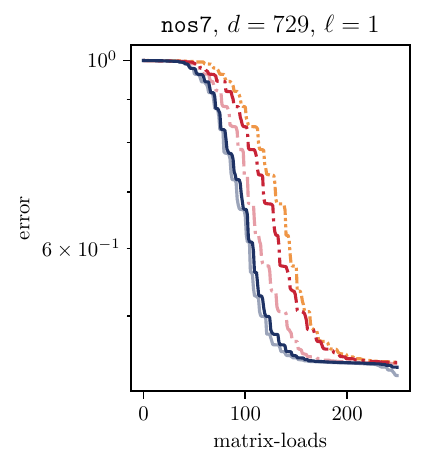}
    \hfil\includegraphics[scale=0.6,trim={.52cm 0 0 0},clip]{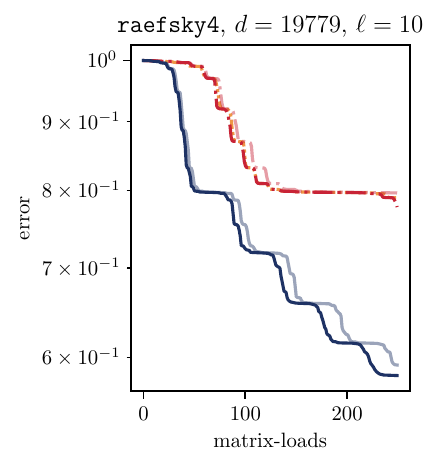}
    \caption{
    Relative error $\|\vec{A}^{-1}\vec{b} - \mathsf{alg}\|_{\vec{A}} / \|\vec{A}^{-1}\vec{b}\|_{\vec{A}}$ versus matrix-loads for
    block-CG 
    ({\protect\raisebox{0mm}{\protect\includegraphics[scale=0.68]{imgs/legend/solid.pdf}}}) with reorthogonalization for 3 iterations, 
    CG 
    ({\protect\raisebox{0mm}{\protect\includegraphics[scale=0.68]{imgs/legend/dashdot.pdf}}}) with reorthogonalization for $3\ell$ iterations, and
    Nystr\"om PCG with $s=3$
    ({\protect\raisebox{0mm}{\protect\includegraphics[scale=0.68]{imgs/legend/dashdotdot.pdf}}}) without any reorthgonalization.
    Light curves show convergence of block-CG and CG with no reorthgonalization.
    The test problems are the same as \cref{fig:iter}.
    }
    \label{fig:iter_fp}
\end{figure}

\section{Omitted proofs}\label{sec:appendix_proofs}
In this section, we provide proofs that were omitted in the main text.

\subsection{Proof of \cref{thm:condno_perturb}}
\label{sec:proofs:nystrompcg}

Let $\vec{A}\langle \vec{K} \rangle$ have rank-$r$ eigendecomposition $\vec{U}\vec{D}\vec{U}^\T$ and define
\begin{equation*}
    \vec{E} = \vec{A} - \vec{A}\langle \vec{K}\rangle
    ,\quad
    \widehat{\vec{A}}_\mu = \vec{A} \langle \vec{K}\rangle + \mu\vec{I}.
\end{equation*}
Note that 
\begin{equation*}
    \vec{P}_{\mu}^{-1}
    = (\theta + \mu) \vec{U}(\vec{D}+\mu\vec{I})^{-1} \vec{U}^\T + (\vec{I} - \vec{U}\vec{U}^\T),
\end{equation*}
and observe that 
\begin{equation*}
    \widehat{\vec{A}}_\mu = \vec{U}\vec{D}\vec{U}^\T + \mu\vec{I}
    = \vec{U}(\vec{D}+\mu\vec{I})\vec{U}^\T + \mu(\vec{I} - \vec{U}\vec{U}^\T).
\end{equation*}

We begin with an upper bound on the eigenvalues of $\vec{P}_\mu^{-1/2}\vec{A}_\mu\vec{P}_\mu^{-1/2}$.
Using the definition of $\vec{E}$ and the triangle inequality,
\begin{equation*}
    \|\vec{P}_\mu^{-1/2}\vec{A}_\mu\vec{P}_\mu^{-1/2}\|
    \leq \|\vec{P}_\mu^{-1/2}\widehat{\vec{A}}_\mu\vec{P}_\mu^{-1/2} \| + \|\vec{P}_\mu^{-1/2}\vec{E}\vec{P}_\mu^{-1/2} \|.
\end{equation*}
Using that $\vec{U}$ has orthonormal columns,
\begin{equation*}
    \vec{P}_\mu^{-1/2} \widehat{\vec{A}}_\mu\vec{P}_\mu^{-1/2}
    = (\theta+\mu) \vec{U}\vec{U}^\T + \mu (\vec{I} - \vec{U}\vec{U}^\T),
\end{equation*}
and hence, since $\theta > 0$, $\|\vec{P}_\mu^{-1/2} \widehat{\vec{A}}_\mu\vec{P}_\mu^{-1/2}\| = \theta + \mu$.
Again using that $\theta > 0$, we note that $\|\vec{P}^{-1/2}\|\leq 1$ and hence $\|\vec{P}_\mu^{-1/2}\vec{E}\vec{P}_\mu^{-1/2} \| \leq \| \vec{E} \|$.
Therefore, we find that 
\begin{equation}
\label{eqn:condno_upper}
    \|\vec{P}_\mu^{-1/2}\vec{A}_\mu\vec{P}_\mu^{-1/2}\|\leq \theta+\mu+\|\vec{E}\|.
\end{equation}

We now derive a lower bound for the eigevnalues of $\vec{P}_\mu^{-1/2}\vec{A}_\mu\vec{P}_\mu^{-1/2}$.
Using that $(\vec{P}_\mu^{-1/2}\vec{A}_\mu\vec{P}_\mu^{-1/2})^{-1} = \vec{P}_\mu^{1/2}\vec{A}_\mu^{-1}\vec{P}_\mu^{1/2}$ is similar to $\vec{A}_\mu^{-1/2} \vec{P}_\mu \vec{A}_\mu^{-1/2}$, the definition of $\vec{P}_\mu$, and the triangle inequality,
\begin{align*}
    \hspace{3em}&\hspace{-3em} 
    \|(\vec{P}_\mu^{-1/2}\vec{A}_\mu\vec{P}_\mu^{-1/2})^{-1}\|
    \\&=\| \vec{A}_\mu^{-1/2} \vec{P}_\mu \vec{A}_\mu^{-1/2} \|
    \\&= \Big\| \vec{A}_\mu^{-1/2} \Big( \frac{1}{\theta + \mu} \vec{U}(\vec{D}+\mu\vec{I}) \vec{U}^\T 
    + (\vec{I} - \vec{U}\vec{U}^\T)\Big) \vec{A}_\mu^{-1/2} \Big\|
    \\&\leq \frac{1}{\theta+\mu} \|\vec{A}_\mu^{-1/2} \vec{U}(\vec{D}+\mu\vec{I})\vec{U}^\T \vec{A}_\mu^{-1/2} \|
    + \| \vec{A}_\mu^{-1/2} (\vec{I} - \vec{U}\vec{U}^\T) \vec{A}_\mu^{-1/2} \|.
\end{align*}
Since $\vec{U}\vec{D}\vec{U}^\T$ is from a Nystr\"om approximation to $\vec{A}$, $\vec{U}\vec{D}\vec{U}^\T\preceq \vec{A}$. Therefore $\vec{U}(\vec{D}+\mu\vec{I})\vec{U}^\T
    \preceq \vec{A} + \mu\vec{I}$ and hence
\begin{equation*}
    \|\vec{A}_\mu^{-1/2} \vec{U}(\vec{D}+\mu\vec{I})\vec{U}^\T \vec{A}_\mu^{-1/2} \|\leq 1.
\end{equation*}
Likewise, since $\vec{I} - \vec{U}\vec{U}^\T \preceq \vec{I}$,
\begin{align*}
    \| \vec{A}_\mu^{-1/2} (\vec{I} - \vec{U}\vec{U}^\T) \vec{A}_\mu^{-1/2} \|
    &\leq \| \vec{A}_\mu^{-1} \| = \frac{1}{\lambda_n+\mu} 
    \\\| \vec{A}_\mu^{-1/2} \vec{P}_\mu \vec{A}_\mu^{-1/2} \|
    &\leq \frac{1}{\theta+\mu}
    + \frac{1}{\lambda_n + \mu}.
\end{align*}
Therefore,
\begin{equation}
    \label{eqn:condno_lower}
    \|(\vec{P}_\mu^{-1/2}\vec{A}_\mu\vec{P}_\mu^{-1/2})^{-1}\|
    \leq \frac{1}{\theta+\mu} + \frac{1}{\lambda_n + \mu}.
\end{equation}
Combining \cref{eqn:condno_upper,eqn:condno_lower} gives the result.
\endproof

\subsection{Bound on Elliptic integral}

\begin{lemma}\label[lemma]{thm:elliptic_k_bd}
For $m<1$, define 
\[
K(m) := \int_0^{\pi/2}(1-m\sin^2(z))^{-1/2}\d{z}.
\]
Then, for all $x > 1$,
\[
\sqrt{x} K(1-x)
\leq 
\frac{5}{8} \log(16 x).
\]
\end{lemma}
\begin{proof}~
Theorem 1.3 in \cite{qiu_vamanamurthy_96} asserts that if $m\in(0,1)$, then
\begin{equation}
\label{eqn:qv96}
    K(m)
    \leq \log\left(\frac{4}{\sqrt{1-m}}\right)\left( 1 + \frac{1-m}{4}\right).
\end{equation}
It's well-known that $K(m)$ satisfies 
\begin{equation*}
    K(m)
    = \frac{1}{\sqrt{1-m}}K\left( \frac{-m}{1-m}\right),
\end{equation*}
and hence
\begin{equation*}
    K(1-x)
    \leq 
    \frac{1}{\sqrt{x}}K\left( \frac{x-1}{x}\right).
\end{equation*}
Rearranging we find that
\begin{equation*}
\sqrt{x} K(1-x)
\leq K\left( 1-\frac{1}{x}\right), 
\end{equation*}
and if $x>1$ then $1-1/x\in(0,1)$ so we can apply \cref{eqn:qv96} to obtain the bound
\begin{equation*}
    \sqrt{x} K(1-x)
    \leq \log\left(\frac{4}{\sqrt{1/x}}\right)\left( 1 + \frac{1/x}{4}\right).
\end{equation*}
The result follows by noting that $(1+1/(4x))/2\leq 5/8$.
\end{proof}

\bibliographystyle{etna}
\bibliography{refs}

\end{document}